\documentclass[letterpaper,11pt,reqno]{amsart}

\usepackage[margin=1.18in]{geometry}
\usepackage{amsmath}
\usepackage{amssymb}
\usepackage{amsthm}
\usepackage{mathtools}
\usepackage{mathrsfs}
\usepackage{enumitem}
\usepackage{microtype}

\usepackage[implicit=true,colorlinks=true,
  linkcolor=blue,citecolor=blue,urlcolor=blue]{hyperref}
\usepackage[notref,notcite,color]{showkeys}
\allowdisplaybreaks[2]

\numberwithin{equation}{section}

\theoremstyle{plain}
\newtheorem{theorem}{Theorem}[section]
\newtheorem{lemma}[theorem]{Lemma}
\newtheorem{proposition}[theorem]{Proposition}

\theoremstyle{definition}
\newtheorem{definition}[theorem]{Definition}

\theoremstyle{remark}
\newtheorem{remark}[theorem]{Remark}

\newcommand{\E}{\mathbb{E}}
\newcommand{\PP}{\mathbb{P}}
\newcommand{\R}{\mathbb{R}}

\newcommand{\NN}{\mathcal{N}}
\newcommand{\ind}{\mathbf{1}}
\newcommand{\cC}{\mathcal{C}}
\newcommand{\cF}{\mathcal{F}}
\newcommand{\cG}{\mathcal{G}}
\newcommand{\cH}{\mathcal{H}}
\newcommand{\cL}{\mathcal{L}}
\newcommand{\cZ}{\mathcal{Z}}
\newcommand{\sZ}{\mathscr{Z}}
\DeclareMathOperator{\Var}{Var}

\DeclareMathOperator{\cum}{cum}
\DeclareMathOperator{\sym}{sym}
\DeclareMathOperator{\Pois}{Poisson}
\DeclareMathOperator{\supp}{supp}
\newcommand{\eps}{\varepsilon}
\newcommand{\noi}{\noindent}
\newcommand{\dd}{d}
\newcommand{\ot}{\otimes}

\newcommand{\Jfour}{\mathcal J_4}
\newcommand{\dW}{d_{\rm W}}
\newcommand{\Del}{\Delta}
\newcommand{\Gam}{\Gamma}
\newcommand{\wt}{\widetilde}
\newcommand{\s}{\sigma}

\newcommand{\inner}[2]{\left\langle #1,#2\right\rangle}

\newcommand{\AND}{\quad{\rm and}\quad}

\hypersetup{
  pdftitle={Gamma approximation and Poisson--Gaussian invariance principle on Poisson chaos},
  pdfauthor={Dionysis Milesis and Guangqu Zheng},
  pdfkeywords={Poisson chaos; Gamma approximation; Poisson--Gaussian invariance principle; four-moment theorem; martingale core; Malliavin--Stein method},
}

\title[Gamma approximation and Poisson--Gaussian invariance]
{Gamma approximation and Poisson--Gaussian \\ invariance principle  on Poisson chaos}

\author[D.~Milesis and G.~Zheng]{Dionysis Milesis and Guangqu Zheng}

\address{
Dionysis Milesis\\
Department of Mathematics and Statistics\\
Boston University\\
665 Commonwealth Avenue\\
Boston, MA 02215, USA
 }

\email{dmilesis@bu.edu}

\address{
Guangqu Zheng\\
Department of Mathematics and Statistics\\
Boston University\\
665 Commonwealth Avenue\\
Boston, MA 02215, USA}

\email{gzheng90@bu.edu}


\begin{document}

\begin{abstract}
We study centered Gamma approximation and a same-kernel
Poisson--Gaussian invariance principle on fixed Poisson chaoses.  For
Gamma approximation, a martingale-core argument extends the
carr\'e-du-champ and $d_2$ estimates of D\"obler and Peccati  (Ann. Probab., 2018)
from
regular kernels to every fourth-integrable chaos element.  In the
diffuse regime, characterized by vanishing fourth add-one energy, this
yields an exact four-moment criterion under uniform integrability of
fourth powers.  In the rare-jump regime, ordinary moments do not
determine the approximation mechanism: convergence of the full moment
sequence may coexist with a nonvanishing fourth add-one energy, and we
construct such centered Gamma limits in every fixed chaos order.

The invariance principle is independent of the Gamma target.  For
Poisson and Gaussian multiple integrals with the same kernel, we bound
both smooth-test discrepancies and the Wasserstein distance in terms of
the variance and the fourth add-one energy.  Thus, vanishing fourth
add-one energy is an intrinsic Lindeberg condition under which the two
chaoses are asymptotically indistinguishable in distribution.  
Combined with a moment-transfer estimate and the Gaussian
fourth-moment theorem, this gives an alternative proof of the
qualitative normal fourth-moment theorem on a fixed Poisson chaos.  A
rainbow example shows that the Lindeberg condition is essential: the
Gaussian analogue may be asymptotically normal while the Poisson
integral converges to a centered compound-Poisson law.  The same
comparison also explains the different behavior of even and odd chaos
orders for diffuse centered Gamma limits.

\end{abstract}

\maketitle

\section{Introduction and main results}
\label{SEC1}

\subsection{From the Gaussian fourth-moment theorem to Gamma limits}
\label{SEC_11}

The fourth-moment theorem of Nualart and Peccati \cite{NP05}
asserts that, for a fixed $q\geq2$, a sequence $(F_n)_n$ in the $q$th
Gaussian Wiener chaos with unit variance satisfies
\noi
\begin{align*}
 F_n\xrightarrow{\mathrm d}N\sim\NN(0,1)
 \quad\Longleftrightarrow\quad
 \E[F_n^4]\longrightarrow3.
\end{align*}
Thus, within a fixed Gaussian Wiener chaos, 
convergence to the normal law is determined by the second and fourth moments.
Peccati and Tudor \cite{PT05} proved the multivariate extension, and
Nourdin and Peccati \cite{NP09a} combined Malliavin calculus with
Stein's method to obtain quantitative bounds.  We refer to
\cite{NP12} for a systematic account and to \cite{APY21} for a recent
survey and applications to random point measures.

A natural noncentral companion to the Nualart--Peccati criterion
was obtained by Nourdin and Peccati
\cite{NP09b}.  For $\nu>0$, let
\noi
\begin{align}
 Z_\nu=2X_{\nu/2,1}-\nu,
 \label{def_Znu}
\end{align}
where $X_{\nu/2,1}$ is Gamma distributed with
shape $\nu/2$ and rate one.
Let $(F_n)_n$ be a sequence in a fixed even Gaussian chaos with
variance converging to $2\nu$.  Nourdin and Peccati established the
following equivalence:

\noi
\begin{align}
  F_n\xrightarrow{\mathrm d} Z_\nu
 \quad\Longleftrightarrow\quad
 \E[F_n^4]-12\E[F_n^3]
 \longrightarrow12\nu^2-48\nu.
 \label{NP_Gam}
\end{align}
The centered Gamma law is asymmetric, so the third moment necessarily
enters the criterion.

The situation on a Poisson chaos is different.  The add-one
Malliavin derivative is a difference operator, and the diffusion chain
rule is no longer available.  The first Poisson Malliavin--Stein bounds
for normal approximation were established in \cite{PSTU10,PZ10};
contraction, cumulant, and Berry--Esseen refinements followed in
\cite{LRP13,ET14,PZ14}.  For centered Gamma approximation, the results
most relevant to the present paper are the following.

\begin{itemize}[leftmargin=2em,itemsep=0.35em]
\item Peccati and Th\"ale \cite{PT13} obtained systematic
Malliavin--Stein bounds for Gamma approximation of Poisson functionals.
They also derived contraction criteria for multiple Poisson integrals,
a quantitative noncentral de Jong theorem for degenerate
$U$-statistics of order two, and multidimensional mixed limits.

\item Fissler and Th\"ale \cite{FT16} asked whether the Gaussian moment
relation \eqref{NP_Gam} forces a Gamma limit on a Poisson chaos.  Their
results in orders two and four require additional contraction
conditions.  In order two, one still assumes the diagonal condition
$\|f_n^2\|\to0$ (see condition (a) in \cite[Theorem 3.5]{FT16}).  
The order-four statement was corrected in
\cite{FT17}, where nonnegative kernels and further vanishing
contractions are imposed.  These results already indicate that
ordinary moments do not capture the full Poisson structure.

\item D\"obler and Peccati \cite{DP18a} developed the centered Gamma
Stein equation on the whole real line.  Their estimates refined
\cite{PT13}, strengthened the Gaussian-space Gamma bounds, and yielded
a noncentral de Jong theorem under minimal uniform integrability
assumptions.

\item D\"obler and Peccati \cite{DP18b} studied normal and centered
Gamma approximation on a general Poisson chaos. In the normal case,
they proved an exact fourth-moment theorem in every chaos order by
combining Stein's method, Mecke identities, and a spectral
carr\'e-du-champ argument. In the Gamma case, they obtained a $d_2$
bound involving the mixed third--fourth moment defect and an additional
fourth add-one term $\Jfour(F)$ (see \eqref{J4F}). In Remark~1.8 of \cite{DP18b}, they pointed out that
``{\it For the time being, it is a challenging open problem to determine whether
such a {\rm(}fourth add-one{\rm)} term can be removed}.''

\end{itemize}

The open problem in Remark~1.8 of \cite{DP18b} is the starting
point of the Gamma part of the present paper.  Theorem~\ref{thm3}
shows that the fourth add-one term cannot be replaced by the mixed
third--fourth moment defect alone.  Theorem~\ref{thm44}{\bf(a)}
extends this obstruction to any prescribed finite collection of
moments.  In the opposite direction, Theorem~\ref{thm44}{\bf(b)}
shows that $\Jfour(F_n)\to0$ is not necessary for centered Gamma
convergence, even when every moment converges.  Thus, $\Jfour$
distinguishes the mechanism of approximation rather than the target
law.

These results further suggest that $\Jfour$ has a role beyond
centered Gamma approximation.  Its vanishing describes a diffuse
regime in which no individual Poisson point has a macroscopic
fourth-order effect.  This naturally raises the question of whether,
in this regime, a Poisson multiple integral behaves like its Gaussian
counterpart.  This question leads to the second main theme of the
paper: a quantitative Poisson--Gaussian invariance principle
controlled by $\Jfour$.

More precisely, given a symmetric kernel $f$, we compare the Poisson
integral $I_q^\eta(f)$ with the Gaussian integral $I_q^W(f)$ built
from the same kernel.  In the homogeneous-sum setting, Nourdin,
Peccati, and Reinert \cite[Theorems~1.2 and~4.1]{NPR10} established
Gaussian universality through low-influence estimates and uniform
moment assumptions on the underlying coordinates.  For discrete
Poisson chaoses, Peccati and Zheng \cite[Theorem~3.4]{PZ14} proved a
qualitative normal universality result under the standing lower bound
on the Poisson intensities in
\cite[Theorem~3.2 and Remark~3.3]{PZ14}.  The transfer principle in
\cite[Proposition~1.5]{DVZ18} shows that a Poisson fourth-moment
condition transfers to the same-kernel Gaussian chaos in the
Lebesgue-control setting.

Our invariance principle instead compares the two laws directly and
does not depend on a prescribed target distribution.  It applies on a
general $\s$-finite control space and allows the cell masses in a
tetrahedral approximation to tend to zero.  The comparison is
quantitative and is controlled by $\Jfour$, which plays the role of
an intrinsic Lindeberg quantity in the Poisson setting.

\subsection{Main results: Gamma approximation and Poisson--Gaussian invariance}
\label{SEC_12}

We now state the main results; the required background is collected
in Section~\ref{SEC2}.  All random objects are defined on a common
probability space $(\Omega,\cF,\PP)$, and $\cL(X)$ denotes the law of
$X$.  Let $(\cZ,\sZ,\mu)$ be a $\s$-finite measure space, let $\eta$ be
a Poisson random measure with control $\mu$, and let $I_q^\eta(f)$
denote the $q$th multiple Wiener--It\^o integral of a symmetric kernel
$f\in L_s^2(\mu^q)$ with respect to $\widehat\eta=\eta-\mu$.
We write
\[
\cC_q^\eta=\{I_q^\eta(f):f\in L_s^2(\mu^q)\}
\]
for the $q$th Poisson chaos.
All analytic estimates and comparison results are stated on this fixed
$\s$-finite control space.  The constructive existence and counterexample
results are realized on suitable atomless $\s$-finite extensions, as
explained at the beginning of Section~\ref{SEC4}.

For a Poisson functional $F\in L^0(\Omega,\s\{\eta\},\PP)$,
its add-one derivative is
$D_zF=F(\eta+\delta_z)-F(\eta)$,
and, whenever $F\in L^4(\Omega, \s\{\eta\}, \PP)$, we set
\begin{align} \label{J4F}
\Jfour(F)
=
\int_\cZ\E[|D_zF|^4]\mu(\dd z).
\end{align}
The quantity $\Jfour(F)$ is the fourth add-one term introduced in
\cite{DP18b}
and it is finite when $F\in\cC_q^\eta\cap L^4(\Omega)$;
see \cite[Theorem~1.8]{Zhe26a}.\footnote{If
$F=I_q^\eta(f)\in L^4(\Omega)$ with $f\in L_s^2(\mu^q)$, then
$f\in L^4(\mu^q)$; see \cite[(1.12)]{Zhe26a}.}
It measures the aggregate fourth-order effect of adding
a single Poisson point.  Accordingly, the condition
$\Jfour(F_n)\to0$ plays the role of an add-one Lindeberg condition: no
individual point has a macroscopic fourth-order effect.

The results below show that $\Jfour$ separates two qualitatively
different mechanisms of centered Gamma approximation.  In the
\emph{diffuse regime}, where $\Jfour(F_n)\to0$, the Poisson chaos
becomes asymptotically close to its same-kernel Gaussian counterpart,
and centered Gamma convergence is governed by a four-moment
criterion.  Outside this regime, genuinely Poissonian rare-jump
effects may persist.  In particular, centered Gamma convergence may
hold, even together with convergence of all moments, while
$\Jfour(F_n)$ stays bounded away from zero.  We refer to this as the
\emph{rare-jump regime}.

When comparing Poisson and Gaussian chaoses, we take, on a product
extension if necessary, an isonormal Gaussian process $W$ over
$L^2(\mu)$ that is independent of $\eta$, and let $I_q^W(f)$ denote
the multiple Wiener integral of the kernel $f$.  We define
\[
\cC_q^W
=
\{I_q^W(f):f\in L_s^2(\mu^q)\},
\]
which is the corresponding Gaussian Wiener chaos of order $q$.

The Gamma results below distinguish the diffuse and rare-jump
mechanisms.  The same-kernel invariance principle is a separate,
target-free result; the two parts meet in the parity criterion of
Proposition~\ref{prop19}.

Let $F$ be centered with $\E[F^2]=2\nu$, and put

\noi
\begin{align}
 \Del_\nu(F)
 =\E[F^4]-12\E[F^3]-12\nu^2+48\nu.
 \label{intro_Del}
\end{align}

Our analytic starting point is a closure of the centered Gamma
estimates in \cite{DP18b} to every
$F\in\cC_q^\eta\cap L^4(\Omega)$.  The constants $C_1(\nu)$ and
$C_2(\nu)$ are defined in \eqref{C_def}.

\begin{theorem} \label{thm1}
Fix an integer $q\geq1$ and
let $F\in\cC_q^\eta\cap L^4(\Omega)$
satisfy
$\E[F^2]=2\nu$.  Then

\noi
\begin{align}
 \frac1{6q}\Del_\nu(F)+\frac1{12q^2}\Jfour(F)
 \leq\Var\left(2F-\frac1q\Gam(F,F)\right)
 \leq\frac13\Del_\nu(F)+\frac1{6q}\Jfour(F).
 \label{intro_Gam_cl}
\end{align}
Here, $\Gam(F,F)$ is the carr\'e-du-champ defined in
\eqref{GF1}--\eqref{GF2}.  Moreover, for the distance $d_2$ defined in
\eqref{d2d},
we have

\noi
\begin{align}
 d_2(F,Z_\nu)
 \leq
 C_1(\nu)\sqrt{|\Del_\nu(F)|}
 +C_2(\nu)
 \left(\frac1q\Jfour(F)\right)^{1/2}.
 \label{Gam_bdd}
\end{align}
\end{theorem}

Theorem~\ref{thm1} extends \cite[Theorem~1.7]{DP18b} by removing
Assumption~$\mathbf A$: the same $d_2$ bound, with the same constants and
the same fourth add-one term, holds for every
$F\in\cC_q^\eta\cap L^4(\Omega)$.  Likewise, \eqref{intro_Gam_cl}
extends \cite[Lemma~5.3]{DP18b} under the sole assumption of a finite
fourth moment.

The fourth add-one term is genuinely structural.  If the moments of
$F_n$ converge to those of $Z_\nu$, then
$\cum_4(F_n)=\E[F_n^4]-3\E[F_n^2]^2\to48\nu$, rather than zero.  Hence
the general estimate
\[
 \Jfour(F)\leq(4q-3)\bigl\{\E[F^4]-3\E[F^2]^2\bigr\},
\]
proved in \cite[Theorem~1.8]{Zhe26a}, yields only boundedness of
$\Jfour(F_n)$ and does not force the second term in \eqref{Gam_bdd} to
vanish.  The smoothing inequality \eqref{dW_d2} converts
\eqref{Gam_bdd} into a Wasserstein bound with a square-root loss.
Proposition~\ref{prop17}(a), already for $q=1$, shows that a general
linear conversion from $d_2$ to $\dW$ is impossible under only a
moment bound.

The proof of Theorem~\ref{thm1} uses the martingale-core approximation
from \cite{Zhe26a}.  Given a finite family of multiple Poisson
integrals, that construction produces a common increasing filtration
generated by finitely many exact Poisson counts.  The corresponding
conditional expectations remain in their original chaoses, have
bounded step kernels with finite-measure support, and converge in every
prescribed $L^p$ norm.  For a fourth-integrable chaos element, the
approximation also converges in the $L^4$ graph norms of all iterated
Malliavin derivatives; see Section~\ref{SEC_23}.  Regular identities
are therefore proved first on finite Charlier chaoses $\cC^\bigstar_q$
(Definition \ref{def_CCQ})
and then passed
to the limit without truncating the chaos expansion.  This closure principle was used in
\cite[Theorem~1.9]{Zhe26a} to remove Assumptions~$\mathbf A$ and
$\mathbf A^{\mathrm{loc}}$ from the D\"obler--Peccati Kolmogorov
bound.  The martingale core and the derivative estimates are developed
in \cite{Zhe26a}, where they are applied to normal approximation in
Kolmogorov distance.  In the present paper, this machinery is used only
as a closure input.  The Gamma closure, the rare-jump constructions,
the same-kernel invariance principle, and the parity analysis are the
subject of the results below.

The first main consequence of Theorem~\ref{thm1} is an
exact four-moment criterion in the \emph{diffuse regime}, i.e.,
when $\Jfour(F_n)\to0$.

\begin{theorem}
\label{thm2}
Fix an integer $q\geq1$ and $\nu\in(0,\infty)$.  Let
$F_n\in\cC_q^\eta\cap L^4(\Omega)$ satisfy
\noi
\begin{align}
 \E[F_n^2]=2\nu
\AND
 \Jfour(F_n)\longrightarrow0.
 \label{diff_ass}
\end{align}
Then $\Del_\nu(F_n)\to0$ implies $F_n\xrightarrow{\mathrm d}Z_\nu$.
If, in addition, $\{F_n^4:n\geq1\}$ is uniformly integrable, then the
following assertions are equivalent:
\begin{align*}
\begin{aligned}
 {\rm(i)}\ F_n\xrightarrow{\mathrm d}Z_\nu
 &\Longleftrightarrow {\rm(ii)}\ d_2(F_n,Z_\nu)\to0\\
 &\Longleftrightarrow {\rm(iii)}\ \E[F_n^3]\to8\nu
 \AND \E[F_n^4]\to12\nu^2+48\nu
 \Longleftrightarrow {\rm(iv)}\ \Del_\nu(F_n)\to0.
\end{aligned}
\end{align*}

\end{theorem}

\begin{remark}\rm
(i)
The sufficient part of Theorem~\ref{thm2} should be compared
with \cite[Theorem~1.7 and Remark~1.8(b)]{DP18b}.  Under their
Assumption~$\mathbf A$, D\"obler and Peccati showed that
\noi
\begin{align*}
\{ \text{$\E[F_n^2]\to2\nu$,
$ \Del_\nu(F_n)\to 0,$
and
 $\Jfour(F_n)\to 0$}\} 
 \quad
 \text{imply}
 \quad \text{$F_n\xrightarrow{\mathrm d}Z_\nu$}.
\end{align*}
They also observed that, for
a fixed chaos order and sufficiently regular kernels, the condition
$\Jfour(F_n)\to0$ follows from the vanishing of the Poisson
contractions $f_n\star_b^a f_n$ with $a<b$;
see also Lemma~\ref{lem_offdiag} for the converse implication.

(ii)
Theorem~\ref{thm2} differs from
\cite[Theorem~1.7]{DP18b} in two respects.  First, by the
closure result of Theorem~\ref{thm1}, no Assumption
$\mathbf A$ or other regularity condition is needed beyond
$F_n\in L^4(\PP)$.  Second, once the diffuse condition
$\Jfour(F_n)\to0$ is imposed, uniform integrability of the fourth
powers turns the sufficient Gamma criterion into an exact
four-moment characterization.  Thus $\Jfour(F_n)\to0$ is best viewed
as a structural, Lindeberg-type condition separating the diffuse
Gamma regime from genuinely Poissonian rare-jump behavior.

(iii) It follows from Proposition~\ref{prop19} that, for odd $q$,
$\Jfour(F_n)\to0$ and $\Del_\nu(F_n)\to0$ cannot hold simultaneously.
Thus the diffuse Gamma criterion is non-vacuous only in even chaos
orders.  This also resolves an  issue raised by Peccati and
Th\"ale \cite[Remark~2.8(iii)]{PT13}, who observed that,
unlike in the Gaussian setting, one could not exclude a priori
odd-order Poisson multiple integrals converging to a centered Gamma
distribution.

(iv) For double Poisson integrals, the diffuse condition admits an
explicit contraction representation.  The contraction notation and
the Poisson product formula are recalled in Section~\ref{SEC_21}.
Indeed, if $F_n=I_2^\eta(f_n)$ with $f_n\in L^2_s(\mu^2)$, then
$D_zF_n=2I_1^\eta(f_n(z,\cdot))$, and we deduce from the first-chaos
case \eqref{I1_prod} of the product formula \eqref{Poi_prod_full},
together with the isometry \eqref{iso}, that
\noi
\begin{align*}
 \Jfour(F_n)
 &=
 48\|f_n\star_2^1 f_n\|_{L^2(\mu)}^2
 +
 16\|f_n\|_{L^4(\mu^2)}^4.
\end{align*}
Consequently,
\noi
\begin{align*}
 \Jfour(F_n)\longrightarrow0
 \quad\Longleftrightarrow\quad
 \|f_n\star_2^1f_n\|_{L^2(\mu)}\longrightarrow0
 \AND
 \|f_n\|_{L^4(\mu^2)}\longrightarrow0.
\end{align*}

This identity clarifies the technical assumptions in
Peccati and Th\"ale \cite[Section~3-(VIII)]{PT13}.  Those assumptions
are primarily integrability conditions ensuring that the relevant
Poisson contractions and product-formula terms are well defined and
square-integrable; they do not by themselves impose asymptotic
smallness.  In contrast, $\Jfour(F_n)\to0$ requires the two genuinely
Poisson diagonal quantities above to vanish.

For $q=2$, the Poisson-specific vanishing contraction conditions
appearing in the Gamma criterion of \cite{PT13} are therefore exactly
summarized by $\Jfour(F_n)\to0$.  The remaining middle-contraction
condition
 $\| f_n\wt{\star_1^1}f_n -f_n\|_{L^2(\mu^2)}
 \to 0$
is the corresponding Gaussian second-chaos Gamma condition.  Thus,
in the double-chaos case, the contraction criterion of
Peccati--Th\"ale can be viewed as the combination of the diffuse
Poisson condition $\Jfour(F_n)\to0$ and the Gaussian Gamma
contraction condition (condition (iv) in \cite[Theorem~1.2]{NP09b}).
\end{remark}

A natural question is whether the diffuse condition
$\Jfour(F_n)\to0$ can itself be deduced from the first four moments.
The next result shows that this is not the case.

\begin{theorem}
\label{thm3}
For every $\nu>0$, there exist an atomless $\s$-finite control
space, a Poisson random measure $\eta$ on it, and
$F_\nu\in\cC_2^\eta$ such that
\noi
\begin{align}
 \E[F_\nu^2]=2\nu,
 \quad
 \E[F_\nu^3]=8\nu,
\AND
 \E[F_\nu^4]=12\nu^2+48\nu,
 \label{exact_four}
\end{align}
but
\noi
\begin{align}
 \cL(F_\nu)\neq\cL(Z_\nu)
 \AND
 \Jfour(F_\nu)>0.
 \label{thm4_2}
\end{align}
\end{theorem}

The obstruction is not repaired by imposing convergence of the
fifth, sixth, or any fixed finite number of moments.  Thus centered
Gamma approximation has a weaker moment principle than the
lattice-constrained Poisson approximation criterion for nonnegative
integer-valued shifts established in \cite{Zhe26b}.  In a different
direction, Theorem~\ref{thm44}{\bf(b)} shows that $\Jfour(F_n)\to0$
is not necessary for centered Gamma convergence: convergence in law
and of every moment may coexist with a positive limiting fourth
add-one energy.

\medskip

We now turn to the second main contribution, which is not tied to the
Gamma target.  We compare a multiple Poisson integral with the
multiple Wiener integral having the same kernel.  The fourth add-one
energy controls the error and yields a quantitative invariance
principle.

\begin{theorem}
\label{thm4}
Fix an integer $q\geq1$.  Let $f\in L_s^2(\mu^q)$ and set
\noi
\begin{align*}
 F=I_q^\eta(f)\in\cC_q^\eta\cap L^4(\Omega),
 \quad
 G=I_q^W(f)\in\cC_q^W,
 \AND
 \sigma^2=\E[F^2]=\E[G^2]
 =q!\|f\|_{L^2(\mu^q)}^2.
\end{align*}
There exists a constant $K_q<\infty$, depending only on $q$, such
that, for every $h\in C^3(\R)$ with bounded derivatives up to order
three,
\noi
\begin{align}
 \left|\E[h(F)]-\E[h(G)]\right|
 \leq
 K_q\|h^{(3)}\|_\infty
 \sqrt{\sigma^2\Jfour(F)}.
 \label{thm5_1}
\end{align}
Moreover, there exists a constant $L_q<\infty$, depending only on
$q$, such that
\noi
\begin{align}
 \dW(F,G)
 \leq
 L_q\bigl\{\sigma^2\Jfour(F)\bigr\}^{1/6},
 \label{thm5_2}
\end{align}
where $\dW$ denotes the $1$-Wasserstein distance as in \eqref{def_dW}.
Consequently, if $F_n=I_q^\eta(f_n)$
and
 $G_n=I_q^W(f_n)$
satisfy
\noi
\begin{align} 
 \sup_n\E[F_n^2]<\infty
 \AND
 \Jfour(F_n)\longrightarrow0,
  \notag
\end{align}
then $\dW(F_n,G_n)\longrightarrow0$.

\end{theorem}

\begin{remark}\rm
(i) Theorem~\ref{thm4} is a target-free Poisson--Gaussian invariance
principle.  Its proof follows the Lindeberg replacement strategy of
Nourdin, Peccati, and Reinert \cite{NPR10}, but the hypotheses and the
conclusion are different.  Their homogeneous-sum comparison is
expressed through low influences and uniform moments of the
standardized coordinates; see \cite[Theorems~1.2 and~4.1]{NPR10}.
The qualitative Poisson universality theorem in
\cite[Theorem~3.4]{PZ14} is formulated for a discrete Poisson chaos
under the standing lower bound on the intensities in
\cite[Theorem~3.2 and Remark~3.3]{PZ14}.  Theorem~\ref{thm4} directly
compares the two laws, applies to every fixed order on a general
$\s$-finite control space, and permits cell masses tending to zero.
For such triangular Poisson coordinates, the standardized moments may
diverge.  We instead keep the cell counts unnormalized and sum the
coordinatewise errors with their cell masses.  The resulting quantity
is exactly controlled by $\Jfour$.

(ii) The condition $\Jfour(F_n)\to0$ cannot be omitted.  The rainbow
kernels in Remark~\ref{rem_inv_fail} have same-kernel Gaussian
integrals converging to a normal law, whereas their Poisson integrals
converge to a centered compound-Poisson law and have fourth add-one
energy bounded away from zero.

(iii) Theorem~\ref{thm4} also forces asymptotic symmetry in odd chaos
orders.  More precisely, if $q$ is odd,
$\sup_n\E[F_n^2]<\infty$, $\Jfour(F_n)\to0$, and $F_n\to Y$ in law,
then $Y$ is symmetric.  Thus, every nonsymmetric limit in an odd
Poisson chaos must arise through a non-diffuse mechanism.
\end{remark}

The exponent $1/6$ in \eqref{thm5_2} results from smoothing a
Lipschitz test function.  Proposition~\ref{prop17} shows that a
Wasserstein estimate of the same square-root order as the smooth-test
bound \eqref{thm5_1} is false, even at fixed positive variance.  The
examples do not determine the optimal exponent in the Wasserstein
bound.

\begin{proposition}
\label{prop17}
The following two assertions hold on suitable atomless $\s$-finite
control spaces.

\begin{enumerate}[label=\textup{(\alph*)},leftmargin=2.4em]

\item
For every integer $q\geq1$, there exist same-kernel pairs
$(F_\lambda,G_\lambda)$, $0<\lambda<1$, in the $q$th Poisson and
Gaussian chaoses such that
\noi
\begin{align}
 \E[F_\lambda^2]=\E[G_\lambda^2]
 &=\lambda,
 \label{prop17_1}\\
 \Jfour(F_\lambda)
 &=q\lambda(1+3\lambda^{1/q})^{q-1},
 \label{prop17_2}
\end{align}
and, as $\lambda\downarrow0$,
\noi
\begin{align}
 d_2(F_\lambda,G_\lambda)
 &\asymp_q\lambda,
 \label{prop17_3}\\
 \dW(F_\lambda,G_\lambda)
 &\asymp_q\sqrt{\lambda}.
 \label{prop17_4}
\end{align}
In particular, neither
$\dW(F,G)\leq C_qd_2(F,G)$ nor
$\dW(F,G)\leq C_q\sqrt{\E[F^2]\Jfour(F)}$
can hold uniformly under only an upper bound on the variance.

\item
Fix $q=2$.  There exist
same-kernel pairs $(F_n,G_n)$ in the second Poisson and Gaussian
chaoses, together with a sequence $\eps_n\downarrow0$, such that

\noi
\begin{align*}
 \E[F_n^2]=\E[G_n^2]=2,
 \quad
 \Jfour(F_n)\asymp\eps_n^4,
 \AND
 \dW(F_n,G_n)\gtrsim\eps_n^{3/2}.
\end{align*}
Consequently,
\noi
\begin{align*}
 \frac{\dW(F_n,G_n)}
 {\sqrt{\E[F_n^2]\Jfour(F_n)}}
 \longrightarrow\infty.
\end{align*}

Thus, a Wasserstein estimate of the same order as
\eqref{thm5_1} fails even at fixed positive variance.
\end{enumerate}
\end{proposition}

The same-kernel comparison also transfers the third and fourth
moments.  The constants $B_q$ and $C_q$ in the next result are
specified in \eqref{Bq_def} and \eqref{Cq_def} below.

\begin{proposition}
\label{prop18}
Fix $q\geq1$ and $\nu>0$.  If
\noi
\begin{align*}
 F=I_q^\eta(f)\in\cC_q^\eta\cap L^4(\Omega),
 \quad
 G=I_q^W(f)\in\cC_q^W,
 \AND
 \E[F^2]=\E[G^2]=2\nu,
\end{align*}
then
\noi
\begin{align}\label{prop18_1}
    \left|\E[F^3]-\E[G^3]\right|
 \leq B_q\sqrt{2\nu\Jfour(F)}
\end{align}
and
\begin{align}\label{prop18_2}
\big|  \sqrt{\E[F^4]}-\sqrt{\E[G^4]} \big|
 \leq C_q\sqrt{\Jfour(F)}.
\end{align}
\end{proposition}

\begin{remark}\rm
Theorem~\ref{thm4} and Proposition~\ref{prop18} also recover the
Poisson-to-Gaussian transfer principle in
\cite[Proposition~1.5]{DVZ18} and give an alternative proof of the
qualitative normal fourth-moment theorem on a fixed Poisson chaos.
Indeed, let $F_n=I_q^\eta(f_n)$ satisfy
$\E[F_n^2]\to1$ and $\E[F_n^4]\to3$, put
$G_n=I_q^W(f_n)$, and set $2\nu_n=\E[F_n^2]$.  Then
\noi
\begin{align*}
 \cum_4(F_n)&\longrightarrow0,
 \quad
 \Jfour(F_n)\leq(4q-3)\cum_4(F_n)\longrightarrow0,\\
 \E[G_n^4]&\longrightarrow3,
 \quad
 G_n\xrightarrow{\mathrm d}N\sim\NN(0,1).
\end{align*}
The first line follows from the moment assumptions and
\eqref{D4_prelim}.  Applying \eqref{prop18_2} with $\nu=\nu_n$
gives the Gaussian fourth-moment convergence, and the Gaussian
fourth-moment theorem yields the last convergence.  Theorem~\ref{thm4}
then gives
$\dW(F_n,G_n)\to0$, and hence $F_n\to N$ in law.  This argument works
on a general $\s$-finite control space.  It is qualitative: inserting
\eqref{D4_prelim} into \eqref{thm5_2} gives a sixth-root bound in the
fourth cumulant, rather than the optimal square-root normal bound in
\cite[Theorem~1.2]{DVZ18}.

For Gamma approximation, Proposition~\ref{prop18} shows that, in the
diffuse regime, the Poisson third--fourth defect follows its Gaussian
counterpart.  In odd order, this leads to the obstruction in
the following result.

\end{remark}

\begin{proposition}\label{prop19}
Fix an integer $q\geq1$ and $\nu>0$.  Let
$F_n=I_q^\eta(f_n)\in\cC_q^\eta\cap L^4(\Omega)$ and
$G_n=I_q^W(f_n)\in\cC_q^W$ satisfy
$\E[F_n^2]=\E[G_n^2]=2\nu$ for every $n\geq1$.

\smallskip
{\rm(i)} Suppose that $\Jfour(F_n)\longrightarrow0$.  Then
\noi
\begin{align}
 \Del_\nu(F_n)-\Del_\nu(G_n)\longrightarrow0.
 \label{prop19_1}
\end{align}
If, in addition, $q$ is even, then the following equivalences hold:
\noi
\begin{align}
 \Del_\nu(F_n)\longrightarrow0
 \quad\Longleftrightarrow\quad
 G_n\xrightarrow{\mathrm d}Z_\nu
 \quad\Longleftrightarrow\quad
 F_n\xrightarrow{\mathrm d}Z_\nu.
 \label{prop19_2}
\end{align}

\smallskip
{\rm(ii)} Suppose that $q$ is odd.  If
$F_n\xrightarrow{\mathrm d}Z_\nu$, then, with $B_q$ as in
\eqref{Bq_def},
\noi
\begin{align}
 \liminf_{n\to\infty}\Jfour(F_n)
 \geq\frac{32\nu}{B_q^2}>0.
 \label{prop19_3}
\end{align}
In particular, centered Gamma convergence in a fixed odd Poisson chaos
cannot occur in the diffuse regime $\Jfour(F_n)\to0$.
\end{proposition}

\begin{remark}\rm
(i) Proposition~\ref{prop19}{\rm(ii)} also yields a uniform metric
separation.  More precisely, there exist
$\eps_{q,\nu}>0$ and $c_{q,\nu}>0$ such that every
$F\in\cC_q^\eta\cap L^4(\Omega)$ with $\E[F^2]=2\nu$ and
$\Jfour(F)\leq\eps_{q,\nu}$ satisfies
\noi
\begin{align*}
 \dW(F,Z_\nu)\geq d_2(F,Z_\nu)\geq c_{q,\nu}.
\end{align*}
Indeed, otherwise one could construct a sequence contradicting
\eqref{prop19_3}.

(ii) The proof also gives the following direct barrier for the mixed
Gamma defect when $q$ is odd:
\noi
\begin{align*}
 \Del_\nu(F)\geq48\nu-12B_q\sqrt{2\nu\Jfour(F)}
 \quad\text{whenever }\E[F^2]=2\nu.
\end{align*}
Consequently, $\Jfour(F_n)\to0$ implies
$\liminf_n\Del_\nu(F_n)\geq48\nu$, whereas
$\Del_\nu(F_n)\to0$ implies
$\liminf_n\Jfour(F_n)\geq8\nu/B_q^2$.
\end{remark}

\begin{remark}\rm
The difference between the two regimes can also be seen directly at
the level of the Gamma carr\'e-du-champ.  In the diffuse regime,
Theorem~\ref{thm2} shows that the Gamma carr\'e-du-champ relation
becomes asymptotically exact.  The rare-jump construction in
Theorem~\ref{thm44}{\bf(b)} behaves differently.  There, a
Gauss--Laguerre discretization first approximates the Gamma L\'evy
measure by finite atomic measures, and the rainbow construction then
realizes the corresponding centered compound-Poisson laws inside a
fixed Poisson chaos; see Sections~\ref{SEC_35} and~\ref{SEC_36}.

More precisely, for the sequence $(R_n)_n$ in
Theorem~\ref{thm44}{\bf(b)},
\noi
\begin{align*}
 R_n\xrightarrow{\mathrm d}Z_\nu,
 \quad
 \E[R_n^k]\longrightarrow\E[Z_\nu^k]
 \quad\text{for every }k\geq1,
 \AND
 \Jfour(R_n)\longrightarrow48q\nu.
\end{align*}
Nevertheless, we deduce from \eqref{intro_Gam_cl} that
\noi
\begin{align*}
 \liminf_{n\to\infty}
 \Var\left(2R_n-\frac1q\Gam(R_n,R_n)\right)
 \geq\frac{4\nu}{q}.
\end{align*}
Thus, even though $R_n$ converges to $Z_\nu$ in law and in all
moments, the diffusion-type Gamma carr\'e-du-champ relation does not
become asymptotically exact.  This makes precise the sense in which
$\Jfour$ distinguishes the mechanism of approximation rather than
the limiting distribution.
\end{remark}

This distinction is visible even for a fixed target and a fixed chaos
order.  In Proposition~\ref{prop_diff_ex}, we construct, for every integer
$d\geq1$, a sequence in the second Poisson chaos converging to $Z_d$
in law and in all moments with $\Jfour\to0$.  By contrast,
Theorem~\ref{thm44}{\bf(b)}, with $q=2$ and $\nu=d$, gives a
sequence with the same limiting law and the same moment convergence
but with $\Jfour\to96d$.

\medskip
{\bf Organization of the paper.}
Section~\ref{SEC2} collects the Poisson and Gaussian chaos notation,
the centered Gamma identities, the distance estimates, and the
martingale-core approximation.  Section~\ref{SEC3} proves the closed
centered Gamma estimates and the diffuse four-moment criterion.
Section~\ref{SEC4} treats the rare-jump regime, including the rainbow
lift, the finite- and all-moment obstructions, and the counterexample
to unrestricted Poisson--Gaussian transfer.  Section~\ref{SEC5} proves
the same-kernel invariance principle, shows that the smooth-test order
does not extend to Wasserstein distance, and establishes the
moment-transfer and Gamma parity results.

\section{Preliminaries}
\label{SEC2}

We collect the Poisson and Gaussian chaos identities, centered Gamma
facts, distance estimates, and martingale-core approximation used in
the later sections.

\subsection{Poisson chaos and Malliavin operators}
\label{SEC_21}

Let $(\cZ,\sZ,\mu)$ be a $\s$-finite measure space,
and let $\eta$ be a
Poisson random measure with intensity measure $\mu$
on a probability space
$(\Omega,\cF,\PP)$.  Let $\widehat\eta=\eta-\mu$ denote
the compensated Poisson random measure.
For $q\geq1$,
let $L_s^2(\mu^q)$ be the real Hilbert space of symmetric kernels
in $L^2(\cZ^q, \mu^q)$, and write $I_q^\eta(f)$ for the multiple Poisson
integral of $f$ with respect to $\widehat\eta$.
The $q$th Poisson Wiener chaos is given by

\noi
\begin{align*}
 \cC_q^\eta=\{I_q^\eta(f):f\in L_s^2(\mu^q)\}
 \AND
 \cC_0^\eta=\R.
\end{align*}

We let $J_q$ denote the orthogonal projection onto $\cC_q^\eta$.
The Wiener--It\^o decomposition is fundamental for stochastic
analysis on Poisson space:

\noi
\begin{align*}
 L^2(\Omega, \s\{\eta\}, \PP)
 =\bigoplus_{q=0}^{\infty}\cC_q^\eta.
\end{align*}
It asserts that every $F\in L^2(\sigma\{\eta\})$ admits an
$L^2$-orthogonal decomposition into multiple Poisson integrals, namely
$F=\E[F]+\sum_{q\geq1}J_qF$,
and the corresponding isometry is

\noi
\begin{align}
 \E[I_p^\eta(f)I_q^\eta(g)]
 =\ind_{\{p=q\}}q!
 \inner{f}{g}_{L^2(\mu^q)},
 \quad
 \forall (f,g)\in L_s^2(\mu^p)\times L_s^2(\mu^q).
 \label{iso}
\end{align}

We refer readers to
the standard references \cite{PT11,Las16,LP18}
for more details.

We use the conventions $L^2(\mu^0)=\R$ and
$I_0^\eta(c)=c$, and interpret an empty product as one.
When no confusion can arise, we write
$\|f\|_{L^2(\mu^p)}=\|f\|_2$ and
$\|X\|_p=\|X\|_{L^p(\Omega)}$.

For a Poisson functional $F=F(\eta)$, the add-one cost derivative
is the difference operator
$D_zF=F(\eta+\delta_z)-F(\eta)$,
where $\delta_z$ denotes the Dirac mass at $z\in\cZ$.
If $F=I_q^\eta(f)$ with $f\in L^2_s(\mu^q)$, then, for
$1\leq r\leq q$,
\noi
\begin{align}
 D^r_{z_1,\ldots,z_r}F
 = \frac{q!}{(q-r)!} I_{q-r}^\eta
 \bigl(f(z_1,\ldots,z_r,\cdot)\bigr).
 \label{iter_D}
\end{align}
In particular, $D^qF=q!\,f$ is deterministic.  This hierarchy of
chaos orders will be used repeatedly in the martingale-core argument.

We deduce from \eqref{iter_D} and the isometry \eqref{iso} that

\noi
\begin{align}
 \int_{\cZ^r}\E[|D^r_{\boldsymbol z}F|^2]
 \,\mu^r(\dd\boldsymbol z)
 =\frac{q!}{(q-r)!}\E[F^2],
 \quad 1\leq r\leq q.
 \label{D_L2}
\end{align}
We next recall the product formula.  Let
$f\in L_s^2(\mu^p)$ and $g\in L_s^2(\mu^q)$, and fix
$0\leq\ell\leq r\leq p\wedge q$.
Writing
\begin{align*}
 \mathbf x\in \cZ^{p-r},\quad
 \mathbf y\in \cZ^{q-r},\quad
 \mathbf z\in \cZ^{r-\ell},
 \AND
 \mathbf u\in \cZ^\ell,
\end{align*}
the contraction $f\star_r^\ell g$ is,
up to the canonical ordering
of its variables,

\noi
\begin{align*}
 &(f\star_r^\ell g)(\mathbf x,\mathbf y,\mathbf z)=
 \int_{\cZ^\ell}
 f(\mathbf x,\mathbf z,\mathbf u)
 g(\mathbf y,\mathbf z,\mathbf u)
 \mu^\ell(\dd\mathbf u).
\end{align*}

Thus, $r$ variables are identified between the two kernels and
$\ell$ of those variables are integrated out.  In particular,
$f\star_r^r g$ is the ordinary contraction
as in \cite[Appendix B]{NP12}, which we also
denote by $f\otimes_r g$.  Whenever all terms are well defined in
$L^2$, the product formula reads

\noi
\begin{align}
 I_p^\eta(f)I_q^\eta(g)
 &=\sum_{r=0}^{p\wedge q}
 r!\binom pr\binom qr
 \sum_{\ell=0}^r\binom r\ell
 I_{p+q-r-\ell}^\eta
 \left(f\wt{\star_r^\ell} g\right).
 \label{Poi_prod_full}
\end{align}
See, for example, \cite[Proposition~5]{Las16}.
Here, $\wt h=\sym(h)$ denotes
the symmetrization of $h$.

We record the contraction estimates used below.  Cauchy--Schwarz
and the fact that symmetrization is an orthogonal projection give
\noi
\begin{align}
 \|h\otimes_s k\|_2
 \leq\|h\|_2\|k\|_2
 \AND
 \|\sym(h)\|_2\leq\|h\|_2,
 \label{contr_CS}
\end{align}
for $h\in L^2(\mu^a)$, $k\in L^2(\mu^b)$, and
$0\leq s\leq a\wedge b$; see also \cite[Appendix B]{NP12}.
More explicitly, let $f\in L_s^2(\mu^q)$,
$0\leq\ell<r\leq q$, and put $t=r-\ell$ and
$f_{\boldsymbol z}=f(\boldsymbol z,\cdot)$ for
$\boldsymbol z\in\cZ^t$.  Fubini's theorem and
\eqref{contr_CS} imply

\noi
\begin{align}
 \big\|f\star_r^\ell f \big\|_2^2
 &=\int_{\cZ^t}
 \|f_{\boldsymbol z}\otimes_\ell
 f_{\boldsymbol z}\|_2^2\,\mu^t(\dd\boldsymbol z)
 \leq\int_{\cZ^t}\|f_{\boldsymbol z}\|_2^4
 \,\mu^t(\dd\boldsymbol z).
 \label{star_section}
\end{align}

For $F=I_q^\eta(f)$ with $f$ symmetric,
we deduce from \eqref{iter_D} and \eqref{iso} that

\noi
\begin{align*}
 \|f_{\boldsymbol z}\|_2^2
 =\frac{(q-t)!}{(q!)^2}
 \E[(D^t_{\boldsymbol z}F)^2].
\end{align*}
Consequently, Jensen's inequality yields
\noi
\begin{align}
 \|f\star_r^\ell f\|_2
 \leq\frac{(q-r+\ell)!}{(q!)^2}
 \left\{
 \int_{\cZ^{r-\ell}}
 \E[|D^{r-\ell}_{\boldsymbol z}F|^4]
 \,\mu^{r-\ell}(\dd\boldsymbol z)
 \right\}^{1/2}.
 \label{star_D}
\end{align}
These inequalities hold with possibly infinite right-hand sides.
For $F\in L^4(\Omega)\cap \cC_q^\eta$, their finiteness follows from
Proposition~\ref{prop_core} below.

We repeatedly use the following elementary cases.  If
$f,g\in L^2(\mu)$ and $fg\in L^2(\mu)$, then
\noi
\begin{align}
 I_1^\eta(f)I_1^\eta(g)
 &=
 I_2^\eta\bigl(\sym(f\ot g)\bigr)
 +I_1^\eta(fg)
 +\inner{f}{g}_{L^2(\mu)}.
 \label{I1_prod}
\end{align}
In particular, for $f\in L^2(\mu)\cap L^4(\mu)$,
\noi
\begin{align*}
 I_1^\eta(f)^2
 &=
 I_2^\eta(f^{\ot2})
 +I_1^\eta(f^2)
 +\|f\|_{L^2(\mu)}^2.
\end{align*}
If $f_1,\ldots,f_q\in L^2(\mu)$ have pairwise disjoint supports,
then
\noi
\begin{align}
 \prod_{k=1}^q I_1^\eta(f_k)
 =
 I_q^\eta\left(
 \sym(f_1\ot\cdots\ot f_q)
 \right).
 \label{disjoint_prod}
\end{align}

More generally, suppose that $f_j\in L_s^2(\mu^q)$ is supported
by $C_j^q$, where $C_1,\ldots,C_m$ are pairwise disjoint.
Then $F_j=I_q^\eta(f_j)$ depends only on the restriction of
$\eta$ to $C_j$, so the $F_j$ are independent and
$\sum_j a_jF_j=I_q^\eta(\sum_j a_jf_j)$ for real $a_j$.
Moreover, $D_zF_j=0$ for $z\notin C_j$.  If each $F_j$
belongs to $L^4(\Omega)$, it follows that
\noi
\begin{align}
 \Jfour\left(\sum_{j=1}^m a_jF_j\right)
 =\sum_{j=1}^m a_j^4\Jfour(F_j).
 \label{J4_blocks}
\end{align}

The Ornstein--Uhlenbeck generator is defined spectrally by
\[
 LF=-\sum_{r\geq1}rJ_rF.
\]
Here,
$\operatorname{dom}L
 =\{H\in L^2(\sigma\{\eta\}):
 \sum_{k\geq1}k^2\|J_kH\|_2^2<\infty\}$.
When $F$, $G$, and $FG$ belong to the domain of $L$, the carr\'e-du-champ
is
\noi
\begin{align}\label{GF1}
 \Gam(F,G)
 =\frac12\{L(FG)-F\,LG-G\,LF\}.
\end{align}
For $F\in\cC_q^\eta\cap L^4(\Omega)$, the square $F^2$ has spectral support in
orders at most $2q$, and we may equivalently define
\noi
\begin{align}\label{GF2}
 \Gam(F,F)
 =\frac12\{L(F^2)+2qF^2\}\in L^2(\PP).
\end{align}

The spectral representation is
\[
 \Gam(F,F)=\sum_{k=0}^{2q}\left(q-\frac k2\right)J_k(F^2)
 \AND
 \E[\Gam(F,F)]=q\E[F^2].
\]
Since the coefficients belong to $[0,q]$, for
$F,H\in\cC_q^\eta\cap L^4(\Omega)$ we have
\noi
\begin{align}
 \|\Gam(F,F)-\Gam(H,H)\|_2
 &\leq q\|F^2-H^2\|_2
 \notag\\
 &\leq q(\|F\|_4+\|H\|_4)\|F-H\|_4.
 \label{Gamma_cont}
\end{align}
This gives the continuity of the spectral carr\'e-du-champ
under fixed-chaos $L^4$ approximation.

For a centered random variable $F$ with finite fourth moment, we use
\noi
\begin{align*}
 \cum_3(F)=\E[F^3]
 \AND
 \cum_4(F)=\E[F^4]-3\E[F^2]^2.
\end{align*}
Recall from \cite[Remark~3.1.6]{Zhe18} that
$\cum_4(F)>0$ for every nonconstant $F$ in a fixed Poisson chaos.
The finite-count approximation below
also yields the quantitative estimate
\noi
\begin{align}
 \Jfour(F)
 \leq(4q-3)\cum_4(F)
 \label{D4_prelim}
\end{align}
for $F\in\cC_q^\eta\cap L^4(\Omega)$.

\subsection{Gaussian chaos}
\label{SEC_gauss_prelim}

On a product extension of the probability space if necessary, let
$W=\{W(h):h\in L^2(\mu)\}$ be an isonormal Gaussian process
independent of $\eta$, meaning that it is a centered Gaussian family
with
covariance structure
\[
\E[W(h) W(\phi)] = \langle h, \phi\rangle_{L^2(\mu)}
\]
for any $h, \phi\in L^2(\mu)$.
For $A\in\sZ$ with $\mu(A)<\infty$, we abbreviate
$W(A)=W(\ind_A)$.  Thus $W(A)$ has law $\NN(0,\mu(A))$,
and the variables $W(A_1),\ldots,W(A_m)$ are independent
when the sets $A_1,\ldots,A_m$ are pairwise disjoint.
We also set $I_0^W(c)=c$.
For $f\in L_s^2(\mu^q)$, we write $I_q^W(f)$ for the $q$th
multiple Wiener integral.  The $q$th Gaussian Wiener chaos is
\noi
\begin{align*}
 \cC_q^W=\{I_q^W(f):f\in L_s^2(\mu^q)\}
 \AND
 \cC_0^W=\R.
\end{align*}
The Wiener--It\^o decomposition of $L^2(\s\{W\})$ is
given by

\noi
\begin{align*}
 L^2(\Omega, \s\{W\}, \PP) =\bigoplus_{q=0}^{\infty}\cC_q^W
\end{align*}
with isometry
\noi
\begin{align}
 \E[I_p^W(f)I_q^W(g)]
 &=\ind_{\{p=q\}}q!\inner{f}{g}_{L^2(\mu^q)}.
 \label{Poi_Gau_iso}
\end{align}
In particular, Poisson and Gaussian integrals of the same symmetric kernel
have the same variance:
\noi
\begin{align*}
 \E[I_q^\eta(f)^2]
 =\E[I_q^W(f)^2]
 =q!\|f\|_{L^2(\mu^q)}^2.
\end{align*}
In view of the Poisson and Gaussian isometries \eqref{iso} and
\eqref{Poi_Gau_iso}, we define the chaos-by-chaos isometry on finite
Gaussian chaos expansions by
\noi
\begin{align}
 \mathsf T\bigl(I_k^W(h)\bigr)=I_k^\eta(h),
 \label{chaos_transport}
\end{align}
and extend it linearly across orthogonal chaos orders.

The product formula in the Gaussian setting
contains only the fully
integrated contractions $f\star_r^r g = f\otimes_r g$
and reads as

\noi
\begin{align}
 I_p^W(f)I_q^W(g)
 =\sum_{r=0}^{p\wedge q}
 r!\binom pr\binom qr
 I_{p+q-2r}^W
 \left( f\wt{\otimes_r} g\right),
 \label{Gau_prod_full}
\end{align}
see, for example, \cite[Theorem~2.7.10]{NP12}.  Therefore, in general,
 $\mathsf T(H^2)\neq(\mathsf T H)^2$.

We will also use Gaussian hypercontractivity in its finite-chaos
form: for $H\in\bigoplus_{k=0}^r\cC_k^W$ and $p\geq2$,

\noi
\begin{align}
 \|H\|_p\leq(p-1)^{r/2}\|H\|_2.
 \label{Gau_hyper}
\end{align}
See, e.g., \cite[Proposition~2.6]{NPR10}.
Thus, $L^2$ convergence of
kernels in a fixed order implies $L^p$ convergence of the
corresponding Gaussian integrals for every finite $p$.
A sequence in a fixed finite sum of Gaussian chaoses with bounded
second moments has uniformly bounded moments of every fixed
order; in particular, each fixed absolute power is uniformly
integrable.  Finally,
$I_q^{-W}(f)=(-1)^qI_q^W(f)$ and $-W$ has the same law as $W$.
Every odd Gaussian chaos is therefore symmetric, and its third
moment is zero.

\subsection{The centered Gamma law}
\label{SEC_22}

Let $Z_\nu$ be defined by \eqref{def_Znu}.
Its cumulant generating
function is

\noi
\begin{align}
 \log\E[e^{tZ_\nu}]
 =-\nu t-\frac\nu2\log(1-2t),
 \quad t<\frac12.
 \label{Gamma_cgf}
\end{align}
Consequently,

\noi
\begin{align*}
 \cum_r(Z_\nu)
 =2^{r-1}(r-1)!\nu,
 \quad r\geq2.
\end{align*}
and in particular

\noi
\begin{align}
 \E[Z_\nu^2]=2\nu,
 \quad
 \E[Z_\nu^3]=8\nu,
 \AND
 \E[Z_\nu^4]=12\nu^2+48\nu.
 \label{Gamma_moments}
\end{align}
Thus, whenever $F$ is centered and $\E[F^2]=2\nu$, the mixed
third--fourth moment defect from \eqref{intro_Del} can be written as

\noi
\begin{align} 
 \Del_\nu(F)
 =\cum_4(F)-12\cum_3(F)+48\nu.
 \notag
\end{align}

For the Gamma approximation bound below, we use the whole-line
centered Gamma Stein equation and the corresponding Stein-factor
estimates from \cite[Theorem~2.3]{DP18a}.

The centered Gamma law is infinitely divisible with L\'evy measure

\noi
\begin{align} \label{Lam_nu}
 \Lambda_\nu(\dd x)
 =\frac\nu2 x^{-1}e^{-x/2}
 \ind_{(0,\infty)}(x)\dd x.
\end{align}
If $\eta_\nu$ is a Poisson random measure with intensity
$\Lambda_\nu$, then $Z_\nu$ has the following first-chaos realization:

\noi
\begin{align*}
 F_\nu^{(1)}
 =\int_0^\infty x\,
 \widehat\eta_\nu(\dd x)
 \stackrel{\mathrm d}=Z_\nu.
\end{align*}
Thus the target itself admits a purely jump realization in the first
Poisson chaos: $F_\nu^{(1)}=I_1^{\eta_\nu}(\phi)$ with $\phi(x)=x$.
For this realization,
\noi
\begin{align*}
 \Jfour(F_\nu^{(1)})
 =\int_0^\infty x^4\Lambda_\nu(\dd x)
 =48\nu.
\end{align*}
If $X$ has moments of all orders, write $m_r=\E[X^r]$ and let
$\kappa_r$ denote its $r$th cumulant.
The moment--cumulant formula is
\begin{align}\label{MCF1}
 m_r
 =
 \sum_{\pi\in\mathcal P(r)}
 \prod_{B\in\pi}\kappa_{|B|},
\end{align}
where $\mathcal P(r)$ denotes the set of partitions of
$\{1,\ldots,r\}$.  Conversely,
\begin{align}\label{MCF2}
 \kappa_r
 =
 \sum_{\pi\in\mathcal P(r)}
 (|\pi|-1)!(-1)^{|\pi|-1}
 \prod_{B\in\pi}m_{|B|};
\end{align}
see \cite[Corollary~3.2.2]{PT11}.
Cumulants are homogeneous and additive over independent sums.
More precisely, for independent $X_1,\ldots,X_m$ with moments
through order $r$ and real $a_1,\ldots,a_m$,
\noi
\begin{align}
 \cum_r\left(\sum_{j=1}^m a_jX_j\right)
 =\sum_{j=1}^m a_j^r\cum_r(X_j).
 \label{cum_add}
\end{align}
For a compensated first Poisson integral, the corresponding
identity is
\noi
\begin{align}
 \cum_r(I_1^\eta(h))
 =\int_{\cZ}h(z)^r\,\mu(\dd z),
 \quad r\geq2,
 \label{cum_I1}
\end{align}
provided $h\in L^2(\mu)\cap L^r(\mu)$; this follows directly
from the Laplace functional of the Poisson process; see
\cite[Theorem~3.9]{LP18}.
These formulas and \eqref{MCF1}--\eqref{MCF2} allow us to pass
between moment and cumulant convergence.

The laws used as moment-method targets below are moment
determinate.  Indeed, a moment generating function finite in
a neighborhood of zero determines the law uniquely.  This
applies to $Z_\nu$ by \eqref{Gamma_cgf}, to Gaussian laws, and
to centered compound Poisson variables with finitely many jump
sizes.  For the last assertion, if
$Y=\sum_{j=1}^m a_j(P_j-\lambda_j)$ with independent
$P_j\sim\Pois(\lambda_j)$, then
$\log\E[e^{tY}]
 =\sum_{j=1}^m\lambda_j
 (e^{a_jt}-1-a_jt)$
 for $t\in\R$.

\subsection{Distances}
\label{SEC_dist}
We use the following distances throughout the paper.

Let

\noi
\begin{align}\label{d2d}
 d_2(X,Y)
 =\sup_{h\in\cH_2}
 |\E[h(X)]-\E[h(Y)]|,
\end{align}

\noi
where $\cH_2$ is the class of continuously differentiable functions
$h:\R\to\R$ satisfying
\noi
\begin{align*}
 \|h'\|_\infty\leq1
 \AND
 |h'(x)-h'(y)|\leq|x-y|.
\end{align*}
We also define the Wasserstein distance
\begin{align}\label{def_dW}
 \dW(X,Y)
 =\sup_{\|h'\|_\infty\leq1}
 |\E[h(X)]-\E[h(Y)]|.
\end{align}
By definition, $d_2(X,Y)\leq\dW(X,Y)$, and convergence in either
$d_2$ or $\dW$ implies convergence in law.
We will also use the following relation from
\cite[Lemma~1.4]{DP18a}.
\begin{lemma}\label{lem_d2_dW}
Let $X$ and $Y$ be random variables with
$\E[|X|]<\infty$ and $\E[|Y|]<\infty$.  Then
\begin{align}\label{dW_d2}
\dW(X,Y) \leq \frac{4}{\sqrt{\pi}}\sqrt{d_2(X,Y)},
\end{align}
whenever $d_2(X,Y)\leq1$.
\end{lemma}

We also use the bounded-Lipschitz distance, defined by
\begin{align*}
 d_{\mathrm{BL}}(X,Y)
 =\sup_{\|h\|_\infty\leq1,\,\|h'\|_\infty\leq1}
 |\E[h(X)]-\E[h(Y)]|.
\end{align*}

The bounded-Lipschitz distance metrizes weak convergence on
$\R$.  For integrable $X,Y$ on a common probability space,
\noi
\begin{align*}
 d_{\mathrm{BL}}(X,Y)
 \leq\dW(X,Y)\leq\E[|X-Y|].
\end{align*}
The same upper bound applies to $d_2$.  The definition of
$\dW$ also gives, for real $a,b$,
\noi
\begin{align}
 \begin{aligned}
 \dW(aX+b,aY+b)&=|a|\dW(X,Y), \quad
 \dW(X+R,Y+R)\leq\dW(X,Y).
 \end{aligned}
 \label{dist_conv}
\end{align}
Here, in the second inequality, $R$ is integrable and independent
of $(X,Y)$.  These statements depend only on the marginal
laws and thus apply to independent convolution on each side.

We record the convergence facts used below.  If
$X_n\to X$ in law and $\{|X_n|:n\geq1\}$ is uniformly
integrable, then $X$ is integrable and
$\dW(X_n,X)\to0$.

More generally, weak convergence together with uniform
integrability of $\{|X_n|^r:n\geq1\}$, for an integer
$r\geq1$, implies convergence of all moments through order
$r$.  For $0<r<s$,
\noi
\begin{align}
 \sup_n\E[|X_n|^r\ind_{\{|X_n|>M\}}]
 \leq M^{r-s}\sup_n\E[|X_n|^s].
 \label{UI_tail}
\end{align}
Thus, a uniform fourth-moment bound implies uniform
integrability of lower powers, but does not in general imply
uniform integrability of the fourth powers themselves.

Finally, if all moments of $X_n$ converge to those of a
moment-determinate random variable $X$, then $X_n\to X$ in
law.  Indeed, the second moments give tightness, and bounded
higher even moments, together with \eqref{UI_tail}, allow
every moment to pass to a subsequential weak limit.  Moment
determinacy identifies each such limit with $X$.  We use this
argument for the targets described in Section~\ref{SEC_22}.

\subsection{Martingale approximation}
\label{SEC_23}

Let us first recall the definition of finite-partition core or Charlier chaos  
 $\cC_q^{\bigstar}$ from \cite[Definition 1.6]{Zhe26a}.
 
 \begin{definition}\label{def_CCQ}
Let $\cC_q^{\bigstar}$ consist of all $I_q^\eta(g)$ such that $g$
is a bounded symmetric step function supported by $C^q$ for some
$C\in\sZ$ with $\mu(C)<\infty$. 
\end{definition}
An element of  $\cC_q^{\bigstar}$
satisfies the regularity assumptions in \cite{DP18b}.
The finite-count projection underlying the martingale core can be
written explicitly.  Let $A_1,\ldots,A_m$ be disjoint sets with
$0<\lambda_i:=\mu(A_i)<\infty$, and let $\Pi_A$ be the orthogonal
projection of $L^2(\mu)$ onto the span of
$\ind_{A_1},\ldots,\ind_{A_m}$:
\[
 \Pi_Ah=\sum_{i=1}^m\frac1{\lambda_i}
 \left(\int_{A_i}h\,\dd\mu\right)\ind_{A_i}.
\]
Charlier orthogonality and the multiple-integral isometry yield
\noi
\begin{align}
 \E[I_q^\eta(f) | \s\{\eta(A_1),\ldots,\eta(A_m)\}]
 =I_q^\eta(\Pi_A^{\otimes q}f).
 \label{count_proj}
\end{align}
Thus, conditioning on finitely many counts preserves the chaos order
and produces a bounded step kernel of finite-measure support; see
\cite[Section~2.2]{Zhe26a}.

We use the following results from
\cite[Theorems~1.7--1.8, Corollary~2.1, and
Proposition~1.10]{Zhe26a}.

\begin{proposition}
\label{prop_core}
Fix $p\in[2,\infty)$ and let
$F=I_q^\eta(f)\in\cC_q^\eta\cap L^p(\Omega)$.  There exist increasing
$\s$-fields
$\cG_n=\s\{\eta(A_{n,1}),\ldots,\eta(A_{n,m_n})\}$, generated by
pairwise disjoint sets of finite positive measure, such that
\[
 F_n=\E[F | \cG_n]=I_q^\eta(f_n)\in\cC_q^{\bigstar}
 \quad\text{and}\quad F_n\longrightarrow F\text{ in }L^p(\Omega).
\]
If $p=4$, the approximants can be chosen so that
\noi
\begin{align}
 \E[|F_n-F|^4]
 +\sum_{r=1}^q\int_{\cZ^r}
 \E[|D^r_{\boldsymbol z}(F_n-F)|^4]
 \,\mu^r(\dd\boldsymbol z)
 \longrightarrow0.
 \label{core_graph}
\end{align}
Consequently,
\noi
\begin{align}
 \Gam(F_n,F_n)&\longrightarrow\Gam(F,F)
 \quad\text{in }L^2(\Omega),
 \label{core_Gamma}\\
 \Jfour(F_n)&\longrightarrow\Jfour(F).
 \label{core_J4}
\end{align}
Moreover, for $1\leq r\leq q$,
\[
 \int_{\cZ^r}\E[|D^r_{\boldsymbol z}F|^4]
 \,\mu^r(\dd\boldsymbol z)
 \leq
 \left\{\prod_{j=0}^{r-1}(4(q-j)-3)\right\}\cum_4(F).
\]
In particular, $f\in L^4(\mu^q)$ and
$\Jfour(F)\leq(4q-3)\cum_4(F)$.
\end{proposition}

The convergence in \eqref{core_Gamma} follows from
\eqref{Gamma_cont}.  The second convergence follows from
$\Jfour(H)^{1/4}=\|DH\|_{L^4(\PP\ot\mu)}$ and the reverse triangle
inequality.  Both conclusions remain valid after replacing $F_n$ by
$a_nF_n$ whenever $a_n\to1$.

\section{The diffuse regime}
\label{SEC3}

We first pass the regular centered Gamma estimates to arbitrary
fourth-integrable fixed-chaos elements by martingale closure, and then
deduce the diffuse four-moment criterion.

\subsection{Closure of the centered Gamma estimates}
\label{SEC_32}

Define
\noi
\begin{align}
 C_1(\nu)
 &=\frac1{\sqrt3}\max\left\{1,\frac2\nu\right\},
 \notag\\
 C_2(\nu)
 &=\frac1{\sqrt6}\max\left\{1,\frac2\nu\right\}
 +\max\left\{
 \sqrt{2\nu},
 \sqrt{\frac2\nu}+\sqrt{\frac\nu2}
 \right\}.
 \label{C_def}
\end{align}

\begin{proof}[Proof of Theorem~\ref{thm1}]
Let $F_n$ be the martingale approximants from
Proposition~\ref{prop_core}, put $\sigma_n^2=\E[F_n^2]$, and, for all
large $n$, set
\[
 a_n=\frac{\sqrt{2\nu}}{\sigma_n}
 \AND
 G_n=a_nF_n.
\]
Since $F_n\to F$ in $L^4(\Omega)$, we have $a_n\to1$,
$\E[G_n^2]=2\nu$, and $G_n\to F$ in $L^4(\Omega)$.  The graph,
carr\'e-du-champ, and fourth add-one convergences in
\eqref{core_graph}--\eqref{core_J4} remain valid under this rescaling.

We may apply \cite[Lemma~5.3]{DP18b} to $G_n$ and obtain
\[
 \frac1{6q}\Del_\nu(G_n)+\frac1{12q^2}\Jfour(G_n)
 \leq\Var\left(2G_n-\frac1q\Gam(G_n,G_n)\right)
 \leq\frac13\Del_\nu(G_n)+\frac1{6q}\Jfour(G_n).
\]
The $L^4$ convergence gives
$\Del_\nu(G_n)\to\Del_\nu(F)$, while \eqref{core_Gamma} and
\eqref{core_J4} give
\[
 2G_n-\frac1q\Gam(G_n,G_n)
 \longrightarrow2F-\frac1q\Gam(F,F)
 \quad\text{in }L^2(\Omega)
 \AND
 \Jfour(G_n)\longrightarrow\Jfour(F).
\]
Passing to the limit proves \eqref{intro_Gam_cl}.

Similarly, \cite[Theorem~1.7]{DP18b}, together with the whole-line
Gamma Stein estimates in \cite[Theorem~2.3]{DP18a}, yields
\eqref{Gam_bdd} with $G_n$ in place of $F$.  Every $h\in\cH_2$ is
one-Lipschitz, and hence
$
 |d_2(G_n,Z_\nu)-d_2(F,Z_\nu)|
 \leq\E[|G_n-F|]\longrightarrow0.
$
The two terms on the right-hand side of \eqref{Gam_bdd} converge by
$L^4$ convergence and \eqref{core_J4}.  Letting $n\to\infty$ proves
the asserted bound.
\end{proof}

\subsection{The diffuse four-moment criterion}
\label{SEC_33}

\begin{proof}[Proof of Theorem~\ref{thm2}]
If $\Del_\nu(F_n)\to0$, then \eqref{Gam_bdd} and \eqref{diff_ass}
give $d_2(F_n,Z_\nu)\to0$, and hence
$F_n\xrightarrow{\mathrm d}Z_\nu$.  This proves the sufficient part.

Assume now that $\{F_n^4:n\geq1\}$ is uniformly integrable.  We prove
the equivalence of \textup{(i)}--\textup{(iv)}.  The implication
\textup{(ii)} $\Rightarrow$ \textup{(i)} is immediate.  Conversely,
\textup{(i)} and the uniform integrability of
$\{|F_n|:n\geq1\}$ imply $\dW(F_n,Z_\nu)\to0$ by
Section~\ref{SEC_dist}; since
$d_2\leq\dW$, this proves \textup{(ii)}.

The uniform integrability of the fourth powers also implies uniform
integrability of all lower powers.  Therefore, \textup{(i)} and
\eqref{Gamma_moments} yield
$\E[F_n^3]\to8\nu$ and
$\E[F_n^4]\to12\nu^2+48\nu$, which is \textup{(iii)}.  We deduce
\textup{(iv)} from \eqref{intro_Del}, and \textup{(iv)} implies
\textup{(ii)} by \eqref{Gam_bdd} and \eqref{diff_ass}.  This completes
the proof.
\end{proof}

\section{The rare-jump regime}
\label{SEC4}

The results in this section are constructive.  We work on an atomless
$\s$-finite control space of infinite mass, so that disjoint sets with
arbitrary prescribed finite measures are available.  For instance,
one may take $(\R_+,\mathcal B(\R_+),\mathrm{Leb})$.

\subsection{An exact second-chaos counterexample}
\label{SEC_34}

Before proving Theorem~\ref{thm3}, we give a one-block
construction.  Let $A$ and $B$ be disjoint
sets with
\noi
\begin{align*}
 \mu(A)=\lambda
 \AND
 \mu(B)=\rho,
\end{align*}
and put
\noi
\begin{align}
 U=\eta(A)-\lambda,
 \quad
 V=\eta(B)-\rho,
 \AND
 G=4UV.
 \label{G_block}
\end{align}
Since $A\cap B=\varnothing$, we deduce from \eqref{I1_prod} that

\noi
\begin{align}\label{block_G}
 G=I_2^\eta\left(
 4\,\sym(\ind_A\ot\ind_B)
 \right)\in\cC_2^\eta.
\end{align}
The kernel is bounded and nonnegative.  The next lemma identifies
parameters for which this block matches the first four centered Gamma
moments.

\begin{lemma}
\label{lem_block}
Let $\theta>0$, and suppose that $\lambda,\rho>0$ satisfy
\noi
\begin{align}
 \lambda\rho=\frac\theta8
 \AND
 \lambda+\rho=\frac16-\frac\theta4.
 \label{lr_eq}
\end{align}
Then the variable $G$ in \eqref{G_block} satisfies
\noi
\begin{align}
 \cum_2(G)=2\theta,
 \quad
 \cum_3(G)=8\theta,
 \AND
 \cum_4(G)=48\theta,
 \label{G_cum}
\end{align}
and
\noi
\begin{align}\label{J4G}
 \Jfour(G)=80\theta-24\theta^2.
\end{align}
Positive solutions of \eqref{lr_eq} exist whenever
\noi
\begin{align}
 0<\theta\leq
 \theta_\ast=\frac{14-8\sqrt3}{3} \approx 0.05.
 \label{theta_star}
\end{align}
\end{lemma}

\begin{proof}
Recall that the first four centered moments of a Poisson variable
with mean $\lambda$ are $0$, $\lambda$, $\lambda$, and
$3\lambda^2+\lambda$.  By independence,

\noi
\begin{align*}
 \E[G^2]&=16\lambda\rho=2\theta,
 \quad \E[G^3]=64\lambda\rho=8\theta,\\
 \E[G^4]
 &=256(\lambda+3\lambda^2)
 (\rho+3\rho^2)
 =256\lambda\rho
 \{1+3(\lambda+\rho)+9\lambda\rho\}
 =12\theta^2+48\theta,
\end{align*}
and \eqref{G_cum} follows.

To prove \eqref{J4G}, observe that
$D_zG=4U\ind_B(z)+4V\ind_A(z)$.  Since $A$ and $B$ are disjoint,
we deduce from \eqref{J4F} that

\noi
\begin{align*}
 \Jfour(G)
 =\int_{\cZ}\E[|D_zG|^4]\,\mu(\dd z)
 &=256\{\lambda(\rho+3\rho^2)
 +\rho(\lambda+3\lambda^2)\}\\
 &=256\lambda\rho\{2+3(\lambda+\rho)\}
 =80\theta-24\theta^2.
\end{align*}
Finally, $\lambda$ and $\rho$ are the roots of
$ x^2- (\frac16-\frac\theta4 )x
 +\frac\theta8=0.$
The discriminant is nonnegative precisely when
$9\theta^2-84\theta+4\geq0$.
The smaller positive root is $\theta_\ast$, and the sum and product in
\eqref{lr_eq} are positive throughout the range in
\eqref{theta_star}.
\end{proof}

We are now ready to prove Theorem~\ref{thm3}.

\begin{proof}[Proof of Theorem~\ref{thm3}]
Let $\theta_\ast$ be as in \eqref{theta_star}, and choose an
integer $m$ so large that $\theta=\nu/m\leq\theta_\ast<1$.
Take $m$ independent copies $G_{\theta,1},\ldots,G_{\theta,m}$ of
the variable $G$ in \eqref{block_G}, with the parameters from
Lemma~\ref{lem_block}, realized over pairwise disjoint regions of
$\cZ$, and set
\noi
\begin{align*}
 F_\nu=\sum_{j=1}^mG_{\theta,j} \in \cC_2^\eta.
\end{align*}

By \eqref{G_cum} and the cumulant additivity formula
\eqref{cum_add},
we have 
$ \cum_2(F_\nu)=2\nu$,
$ \cum_3(F_\nu)=8\nu$,
and
$ \cum_4(F_\nu)=48\nu.$
Then, we can deduce  \eqref{exact_four}  from  \eqref{MCF1} and
\eqref{Gamma_moments}.

The variable $F_\nu$ has a discrete law, whereas $Z_\nu$ has a
continuous law, so $\cL(F_\nu)\neq\cL(Z_\nu)$.  Since the blocks are
supported on pairwise disjoint regions, we deduce from
\eqref{J4_blocks} and \eqref{J4G} that
$ \Jfour(F_\nu)
 =m(80\theta-24\theta^2)
 =80\nu-\frac{24\nu^2}{m}>0.
$
This proves \eqref{thm4_2} and completes the proof.
\end{proof}

\subsection{Rainbow lifts of compound-Poisson laws}
\label{SEC_35}

We first record the elementary centered-Poisson moment estimates used
throughout this construction.  Let $P_a\sim\Pois(a)$ and put
$U_a=P_a-a$.  For every $r\geq2$, the $r$th cumulant of $U_a$ equals
$a$.  Hence, for each fixed $r\geq2$,

\noi
\begin{align}
 M_r(a)=\E[U_a^r]
 =a+O_r(a^2),
 \quad a\downarrow0.
 \label{Poi_mom}
\end{align}
The first values needed below are

\noi
\begin{align*}
 M_2(a)=a,
 \quad
 M_3(a)=a,
 \AND
 M_4(a)=a+3a^2.
\end{align*}

Fix $q\geq1$, $c>0$, and $t>0$.  For $m\geq1$, put
$ a_m=\left( t/m\right)^{1/q}$.
Choose pairwise disjoint sets $A_{j,k}^{(m)}\in\sZ$,
$1\leq j\leq m$ and $1\leq k\leq q$, such that
$ \mu(A_{j,k}^{(m)})=a_m$,
and set

\noi
\begin{align*}
 N_{j,k}^{(m)}
 &=\eta(A_{j,k}^{(m)})
 \sim\Pois(a_m),
 \quad
 1\leq j\leq m\AND 1\leq k\leq q,\\
 U_{j,k}^{(m)}
 &=N_{j,k}^{(m)}-a_m
 =I_1^\eta\bigl(\ind_{A_{j,k}^{(m)}}\bigr).
\end{align*}
By construction,
the random variables $N_{j,k}^{(m)}$ are independent.

Now define the \textit{$q$-color rainbow sum}
\noi
\begin{align}
 H_m^{(q)}(c,t)
 =c\sum_{j=1}^m
 \prod_{k=1}^qU_{j,k}^{(m)}.
 \label{rainbow_def}
\end{align}
The terminology refers to the fact that every monomial uses one
coordinate from each of $q$ disjoint `colors'.  By \eqref{disjoint_prod},
\noi
\begin{align*}
 \prod_{k=1}^qU_{j,k}^{(m)}
 =I_q^\eta\left(
 \sym\bigl(
 \ind_{A_{j,1}^{(m)}}\ot\cdots\ot
 \ind_{A_{j,q}^{(m)}}
 \bigr)
 \right).
\end{align*}
Thus $H_m^{(q)}(c,t)\in\cC_q^\eta$, with a bounded nonnegative step kernel
of finite-measure support.

\begin{proposition}
\label{prop_rainbow}
For every $r\geq2$,

\noi
\begin{align}
 \cum_r\bigl(H_m^{(q)}(c,t)\bigr)
\xrightarrow{m\to\infty} tc^r.
 \label{rainbow_cum}
\end{align}

Consequently, with $P_t\sim\Pois(t)$,
\noi
\begin{align}
 H_m^{(q)}(c,t)&\xrightarrow{\mathrm d}c(P_t-t),
 \notag\\
 \E\left[H_m^{(q)}(c,t)^r\right]
 &\longrightarrow\E\left[c(P_t-t)^r\right],
 \quad r\geq1.
 \label{rainbow_mom}
\end{align}
Moreover,
\noi
\begin{align}
 \Jfour\bigl(H_m^{(q)}(c,t)\bigr)
 =qtc^4(1+3a_m)^{q-1}
 \xrightarrow{m\to\infty} qtc^4.
 \label{rainbow_J4}
\end{align}

\end{proposition}

\begin{proof}
Let $U_{a,1},\ldots,U_{a,q}$ be independent centered Poisson
variables of mean $a$, and put
$
 W_a=\prod_{k=1}^qU_{a,k}.
$
By \eqref{Poi_mom}, for each fixed $r\geq2$,

\noi
\begin{align}\label{WAR}
 \E[W_a^r]
 =M_r(a)^q
 =a^q+O_r(a^{q+1}).
\end{align}

Since $\E[W_a]=0$, every partition containing a singleton block
has zero contribution in the moment--cumulant formula \eqref{MCF2}.
The one-block partition contributes to \eqref{WAR}.
Every other nonvanishing partition has at least two blocks, each of
size at least two.  Since, for every fixed $s\geq2$,
 $\E[W_a^s]=O_s(a^q)$,
a partition with $k\geq2$ blocks contributes $O_r(a^{qk})$, and hence
$O_r(a^{2q})$.  Consequently,
\[
 \cum_r(W_a)
 =
 a^q+O_r(a^{q+1}+a^{2q}).
\]

Since the $m$ blocks in \eqref{rainbow_def} are independent, we deduce from
\eqref{cum_add} that
\noi
\begin{align*}
 \cum_r\bigl(H_m^{(q)}(c,t)\bigr)
 =mc^r\cum_r(W_{a_m})
 \longrightarrow tc^r,
\end{align*}
which proves \eqref{rainbow_cum}.  The limiting cumulants are those of
$c(P_t-t)$ by \eqref{cum_I1}.  It follows from \eqref{MCF1} that every
fixed moment converges.  Since the centered Poisson law is
moment-determinate, the moment method recorded in Section~\ref{SEC_dist}
proves the convergence in law and therefore \eqref{rainbow_mom}.

If $z\in A_{j,k}^{(m)}$, then
$D_zH_m^{(q)}(c,t)=c\prod_{\ell\neq k}U_{j,\ell}^{(m)}$.
Integrating over the $qm$ cells and using $M_4(a)=a+3a^2$, we obtain
\[
 \Jfour\bigl(H_m^{(q)}(c,t)\bigr)
 =qma_mc^4M_4(a_m)^{q-1}
 =qtc^4(1+3a_m)^{q-1},
\]
which proves \eqref{rainbow_J4}.
\qedhere

\end{proof}

\begin{remark}\rm
\label{rem_inv_fail}
The rainbow kernels also show that the diffuse assumption in
Theorem~\ref{thm4} is essential.  Let $G_m^{(q)}(c,t)$ denote the
same-kernel Gaussian integral associated with
$H_m^{(q)}(c,t)$.  Since the sets $A_{j,k}^{(m)}$ are pairwise
disjoint, we may write
$W(A_{j,k}^{(m)})=\sqrt{a_m}Z_{j,k}$ with the $Z_{j,k}$ independent
standard normal random variables.  The products
$\prod_{k=1}^qZ_{j,k}$ are i.i.d., centered, and have variance one.
Hence, by the classical central limit theorem,
\noi
\begin{align*}
 G_m^{(q)}(c,t)
 &=c\sqrt{\frac{t}{m}}
 \sum_{j=1}^m\prod_{k=1}^qZ_{j,k}
 \xrightarrow{\mathrm d}\NN(0,c^2t).
\end{align*}
On the Poisson side, Proposition~\ref{prop_rainbow} gives
\noi
\begin{align*}
 H_m^{(q)}(c,t)\xrightarrow{\mathrm d}c(P_t-t)
 \AND
 \Jfour\bigl(H_m^{(q)}(c,t)\bigr)\longrightarrow qtc^4>0.
\end{align*}
Thus, normal convergence of a tetrahedral Gaussian chaos does not by
itself imply normal convergence of its Poisson analogue.  The usual
uniform-moment hypothesis from \cite{NPR10} also fails here: for
$P_{a_m}\sim\Pois(a_m)$,
\[
 \E\left[
 \left|\frac{P_{a_m}-a_m}{\sqrt{a_m}}\right|^3
 \right]\longrightarrow\infty,
\]
as follows already from the event $\{P_{a_m}=1\}$.  Since
$a_m\to0$, this triangular array also lies outside the intensity
lower-bound regime of \cite[Theorem~3.4]{PZ14}; see
\cite[Theorem~3.2 and Remark~3.3]{PZ14}.  The fourth add-one energy
therefore acts as a genuine Poisson
Lindeberg condition, rather than as a technical remainder in the
proof of Theorem~\ref{thm4}.
\end{remark}

\subsection{Finite and infinite collections of moments}
\label{SEC_36}

The rainbow construction \eqref{rainbow_def} reduces the problem to choosing finite
compound-Poisson laws with prescribed Gamma moments.  We do so
by discretizing the Gamma L\'evy measure with Gauss--Laguerre
quadrature.

We use the $K$-point Gauss--Laguerre rule in the normalization of
\cite[Section~10.5]{QSS07}.  Let $\{L_n\}$ denote the unnormalized
Laguerre polynomials defined by
\[
 L_n(x)=e^x\frac{\dd^n}{\dd x^n}(e^{-x}x^n),
 \quad n\geq0.
\]
These are orthogonal polynomials on $\R_+$ with respect to the weight
$e^{-x}$.  For every polynomial $p$ of degree at most $2K-1$,
\begin{align}\label{GL_exp}
 \int_0^\infty e^{-x}p(x)\,\dd x
 =\sum_{i=1}^K\omega_{K,i}p(y_{K,i}),
\end{align}
where $0<y_{K,1}<\cdots<y_{K,K}$ are the zeros of $L_K$ and
\noi
\begin{align*}
 \omega_{K,i}
 =\frac{(K!)^2y_{K,i}}{L_{K+1}(y_{K,i})^2},
 \quad 1\leq i\leq K,
\end{align*}
see \cite[(10.41)]{QSS07}.  In particular,

\noi
\begin{align}
 \sum_{i=1}^K\omega_{K,i}y_{K,i}^\ell
 =\ell!,
 \quad 0\leq\ell\leq2K-1.
 \label{GL_mom}
\end{align}
The centered Gamma law $Z_\nu$ is moment-determinate,
and we will
combine this fact with \eqref{GL_mom} to approximate its L\'evy measure $\Lambda_\nu$ in \eqref{Lam_nu}
 by finite atomic measures.

With $y_{K,i},\omega_{K,i}$ as in \eqref{GL_exp}, we put

\noi
\begin{align}
 c_{K,i} =2y_{K,i}
\AND
 t_{K,i} =\frac{\nu\omega_{K,i}}{c_{K,i}}
 \label{ct_def}
\end{align}
for $1\leq i\leq K$.
Let $P_{K,1},\ldots,P_{K,K}$ be independent Poisson variables with
means $t_{K,i}$, and define

\noi
\begin{align}\label{def_YK}
 Y_K=\sum_{i=1}^Kc_{K,i}
 (P_{K,i}-t_{K,i}).
\end{align}

\begin{lemma}
\label{lem_YK}
Let $Z_\nu$ be as in \eqref{def_Znu}.
For every $2\leq r\leq2K$, we have

\noi
\begin{align*}
 \cum_r(Y_K)
 =2^{r-1}(r-1)!\nu
 =\cum_r(Z_\nu).
\end{align*}
Consequently, the first $2K$ moments of $Y_K$ and $Z_\nu$
coincide, but $\cL(Y_K)\neq\cL(Z_\nu)$.

\end{lemma}

\begin{proof}
Since cumulants are additive over independent sums, we deduce from
\eqref{ct_def} and \eqref{GL_mom} that

\noi
\begin{align}\label{CUMR}
 \cum_r(Y_K)
 =\sum_{i=1}^Kt_{K,i}c_{K,i}^r
 =\nu\,2^{r-1}\sum_{i=1}^K\omega_{K,i}y_{K,i}^{r-1}
 =2^{r-1}(r-1)!\nu
\end{align}
for $r\leq2K$.
The equality of the first $2K$ moments now follows from the
moment--cumulant formula \eqref{MCF1}.

Taking $\ell=0$ in \eqref{GL_mom} and using \eqref{ct_def}, we obtain
\noi
\begin{align*}
 \sum_{i=1}^Kt_{K,i}c_{K,i}
 =\nu\sum_{i=1}^K\omega_{K,i}=\nu.
\end{align*}
On the positive-probability event that all $P_{K,i}$ vanish, we have
$Y_K=-\nu$.  Thus $Y_K$ has an atom at $-\nu$, whereas the centered
Gamma law is continuous.  This completes the proof.
\end{proof}

\begin{theorem}
\label{thm44}
Fix an integer $q\geq1$ and $\nu\in(0,\infty)$.  On a suitable
atomless $\s$-finite Poisson space, the following assertions hold.

{\rm\bf (a)} For every finite integer $M\geq4$, there is a sequence
$F_n\in\cC_q^\eta$ such that $\E[F_n^2]=2\nu$ and
\noi
\begin{align}
 \E[F_n^r]
 &\longrightarrow\E[Z_\nu^r],
 \quad 1\leq r\leq M,
 \label{M_mom}\\
 F_n&\xrightarrow{\mathrm d}Y_M,
 \quad
 \cL(Y_M)\neq\cL(Z_\nu),
 \label{M_law}\\
 \Jfour(F_n)&\longrightarrow48q\nu > 0.
 \notag
\end{align}

{\rm\bf (b)} There is a sequence
$R_n\in\cC_q^\eta$ such that
\noi
\begin{align}
& R_n\xrightarrow{\mathrm d}Z_\nu
 \AND
 \E[R_n^r]\longrightarrow\E[Z_\nu^r]
 \quad\text{for every }r\geq1,
 \notag\\
&\text{while}\quad\Jfour(R_n)\longrightarrow48q\nu>0.
 \label{all_J4}
\end{align}

\end{theorem}

\begin{proof}[Proof of Theorem~\ref{thm44}\rm\textbf{(a)}]
Fix $M\geq4$ and choose
\noi
\begin{align*}
 K\geq\left\lceil\frac M2\right\rceil \geq 2.
\end{align*}
For each $i$, take an independent rainbow lift
$H_m^{(q)}(c_{K,i},t_{K,i})$ as in \eqref{rainbow_def}
with $c_{K,i},t_{K,i}$ as in \eqref{ct_def},
using mutually disjoint regions of the
control space $\cZ$, and set
\noi
\begin{align}
 F_m^{(q,K)}
 =\sum_{i=1}^K
 H_m^{(q)}(c_{K,i},t_{K,i}),
 \label{Fm_qK}
\end{align}
which belongs to $\cC_q^\eta$.

By Proposition~\ref{prop_rainbow}, \eqref{rainbow_mom}, and
independence of the $K$ summands,
$F_m^{(q,K)}\xrightarrow{\mathrm d}Y_K$, and every fixed moment
converges to the corresponding moment of $Y_K$.  Set $Y_M=Y_K$.
Lemma~\ref{lem_YK} therefore yields
\eqref{M_mom} and \eqref{M_law}.

Since $\cum_2(W_{a_m})=a_m^q=t/m$, independence and
\eqref{CUMR} with $r=2$ give
\[
 \E[(F_m^{(q,K)})^2]
 =\sum_{i=1}^Kt_{K,i}c_{K,i}^2
 =2\nu.
\]
The regions used for different $i$ are disjoint.  We deduce from
\eqref{J4_blocks}, Proposition~\ref{prop_rainbow}, and
\eqref{CUMR} with $r=4$ that
\[
 \Jfour(F_m^{(q,K)})
 \longrightarrow q\sum_{i=1}^Kt_{K,i}c_{K,i}^4
 =48q\nu.
\]
Relabeling $m$ as $n$ proves part~{\bf(a)}.
\qedhere

\end{proof}

We now let the number of quadrature points increase.

\begin{lemma}
\label{lem_YK_Gam}
Let $Y_K$ be as in \eqref{def_YK}.
As $K\to\infty$,
\noi
\begin{align}
 Y_K\xrightarrow{\mathrm d}Z_\nu,
 \label{YK_Gam}
\end{align}
and, for every fixed $r\geq1$,
\noi
\begin{align}
 \E[Y_K^r]=\E[Z_\nu^r]
 \quad\text{for all sufficiently large }K.
 \label{YK_eve}
\end{align}
\end{lemma}

\begin{proof}

The moment statement follows directly from Lemma~\ref{lem_YK}.  In
particular, $\E[Y_K^2]=2\nu$, so $(Y_K)_K$ is tight.  Let
$(Y_{K_j})_j$ be a weakly convergent subsequence, with limit $Y$.
For every fixed $r$, the moments of order $2r$ eventually agree with
those of $Z_\nu$; hence $(|Y_{K_j}|^r)_j$ is uniformly integrable by
\eqref{UI_tail}.  Therefore
$\E[Y^r]=\E[Z_\nu^r]$
for every $r\geq1$.
The centered Gamma law is moment-determinate by \eqref{Gamma_cgf};
hence $Y=Z_\nu$ in law.  Every weakly convergent
subsequence has the same limit, proving \eqref{YK_Gam}.
\end{proof}

\begin{proof}[Proof of Theorem~\ref{thm44}{\rm\bf(b)}]
For each $K$, the sequence $F_m^{(q,K)}$ in \eqref{Fm_qK} converges in
law and in every fixed moment to $Y_K$, while its add-one energy
converges to $48q\nu$.  Choose a diagonal sequence $m_K\to\infty$ so
large that
\noi
\begin{align}
 d_{\mathrm{BL}}
 \bigl(F_{m_K}^{(q,K)},Y_K\bigr)&\leq K^{-1},
 \label{diag_BL}\\
 \left|
 \E[(F_{m_K}^{(q,K)})^r]-\E[Y_K^r]
 \right|&\leq K^{-1},
 \quad1\leq r\leq K,
 \label{diag_mom}\\
 \left|
 \Jfour(F_{m_K}^{(q,K)})-48q\nu
 \right|&\leq K^{-1}.
 \label{diag_J4}
\end{align}
Set
$R_K=F_{m_K}^{(q,K)}.$
We deduce from Lemma~\ref{lem_YK_Gam} and \eqref{diag_BL} that
$R_K\xrightarrow{\mathrm d}Z_\nu$.  For each fixed $r$, it follows from
\eqref{diag_mom} and \eqref{YK_eve} that
$\E[R_K^r]\to\E[Z_\nu^r]$.  Finally, \eqref{all_J4} follows from
\eqref{diag_J4}.

This completes the proof of part~{\bf(b)}.
\qedhere

\end{proof}

\subsection{Diffuse and rare-jump realizations}
\label{SEC_37}

The positive-energy limit in Theorem~\ref{thm44}{\bf(b)} is not a
universal Gamma value of $\Jfour$.  The same target can also be reached
through a diffuse sequence.  We record a simple construction.

\begin{proposition}
\label{prop_diff_ex}
Let $d\geq1$ be an integer.  There is a
sequence $Q_n\in\cC_2^\eta$ such that
\noi
\begin{align*}
 Q_n\xrightarrow{\mathrm d}Z_d,
 \quad
  \Jfour(Q_n)\longrightarrow0,
  \AND
 \E[Q_n^r] \longrightarrow\E[Z_d^r]
 \quad\text{for every }r\geq1.
\end{align*}
\end{proposition}

\begin{proof}
Take the control space
$\{1,\ldots,d\}\times\R_+$ with counting measure times Lebesgue
measure, and let
$ A_{n,j}=\{j\}\times(0,n]$
for $1\leq j\leq d.$
Thus, for each $n$, the sets $A_{n,1},\ldots,A_{n,d}$ are disjoint
and have measure $n$.  Put
\noi
\begin{align*}
 g_{n,j}=n^{-1/2}\ind_{A_{n,j}}
 \AND
 X_{n,j}=I_1^\eta(g_{n,j}).
\end{align*}

By \eqref{cum_I1},
$ \cum_2(X_{n,j})=1$
and
$ \cum_r(X_{n,j})=n^{1-r/2}\longrightarrow0$
for $r\geq3$.
It follows from \eqref{MCF1} that all moments converge to those of a
standard normal variable.  Since the coordinates are independent, all
joint moments factor and converge to the corresponding moments of a
vector $(N_1,\ldots,N_d)$ of independent standard normal variables.
Moment determinacy of the Gaussian product law then yields joint
convergence in distribution.  Set
\noi
\begin{align*}
 Q_n=\sum_{j=1}^d I_2^\eta(g_{n,j}^{\ot2}).
\end{align*}

We deduce from the first-chaos product formula \eqref{I1_prod} that
\noi
\begin{align*}
 I_2^\eta(g_{n,j}^{\ot2})
 =X_{n,j}^2-I_1^\eta(g_{n,j}^2)-1
 =X_{n,j}^2-\frac{X_{n,j}}{\sqrt n}-1.
\end{align*}
Consequently,
$
 Q_n\xrightarrow{\mathrm d}
 \sum_{j=1}^d(N_j^2-1)
 \stackrel{\mathrm d}=Z_d,
$
and the same polynomial representation yields convergence of every
moment.

For $z\in A_{n,j}$,
$D_zQ_n=2g_{n,j}(z)X_{n,j}$.
Since $\|g_{n,j}\|_{L^4(\mu)}^4=n^{-1}$ and
$\E[X_{n,j}^4]=3+n^{-1}$,
we have
$\Jfour(Q_n)=16d n^{-1}(3+n^{-1})\to0$.  This completes the
proof.
\end{proof}

\begin{remark}\rm \label{rem_J4}
(i)
Combining Proposition~\ref{prop_diff_ex} with
Theorem~\ref{thm44}{\bf(b)}, we see that, for the same order $q=2$ and the
same target $Z_d$, either limit can occur:
\noi
\begin{align*}
 \Jfour(Q_n)\longrightarrow0
 \quad\text{or}\quad
 \Jfour(R_n)\longrightarrow96d.
\end{align*}
Thus, $\Jfour$ records the mechanism of convergence rather than the
limiting law.

(ii) The rare-jump construction also explains why ordinary moments cannot
replace $\Jfour$.  The Gauss--Laguerre rule discretizes the Gamma L\'evy
measure at finitely many jump sizes.  The rainbow lift then embeds each
centered compound-Poisson summand into any prescribed fixed chaos
order.
As the atomic measures approximate the L\'evy measure, the laws
converge to Gamma, but the fourth add-one energy records the aggregate
fourth power of the jumps and converges to $48q\nu$.
By contrast, when $\Jfour(F_n)\to0$, all individual add-one effects
are negligible at fourth order, and Theorem~\ref{thm2}
reduces Gamma approximation to the mixed moment defect
$\Del_\nu(F_n)$.  The next section explains why this regime is
close to the Gaussian one and why the parity of the chaos order
matters.
\end{remark}

\section{Poisson--Gaussian invariance and Gamma parity}
\label{SEC5}

This section proves Theorem~\ref{thm4},
Proposition~\ref{prop17}, Proposition~\ref{prop18}, and
Proposition~\ref{prop19}.  The first part is independent of the Gamma
target: we approximate an arbitrary kernel by tetrahedral step kernels
and compare the corresponding Poisson and Gaussian multilinear forms
through a coordinatewise Lindeberg replacement.  We then prove the
separation statements in Proposition~\ref{prop17}.  The remaining
subsections propagate $\Jfour$ through iterated derivatives, control
the contractions absent from the Gaussian product formula, and derive
the moment transfers and the Gamma parity results.

Let $q\geq1$ be fixed and let $f\in L_s^2(\mu^q)$.  Set
\noi
\begin{align}\label{FG_1}
 F=I_q^\eta(f)\in\cC_q^\eta\cap L^4(\Omega)
 \AND
 G=I_q^W(f)\in\cC_q^W.
\end{align}

\subsection{The Poisson--Gaussian invariance principle}
\label{SEC_50}

The proof of Theorem~\ref{thm4} has two steps.  We first reduce to
bounded tetrahedral step kernels and then replace the Poisson cell
counts by Gaussian increments one coordinate at a time.  A symmetric
step kernel is called \emph{tetrahedral} relative to a finite partition
$A_1,\ldots,A_N$ if it vanishes whenever two of its arguments belong
to the same partition cell.
For the proof of the invariance principle, we may assume without loss
of generality that the control measure is nonatomic.  Indeed, we use
the standard marking enlargement
$(\cZ\times[0,1],\sZ\otimes\mathcal B([0,1]),\mu\otimes\dd u)$
and lift the kernel in the natural way.  This preserves the laws of
the corresponding Poisson and Gaussian multiple integrals as well as
$\Jfour$.

\begin{lemma}
\label{lem_tetra_core}
Let $F$ and $G$ be as in \eqref{FG_1}.  After passing,
if necessary, to the marked control space
$(\cZ\times[0,1],\mu\otimes\dd u)$, there are bounded symmetric
tetrahedral step kernels $f_m$ of finite-measure support such that,
with $F_m=I_q^\eta(f_m)$ and $G_m=I_q^W(f_m)$,
\noi
\begin{align}
 \|F_m-F\|_{L^4(\Omega)}
 +\|G_m-G\|_{L^4(\Omega)}
 &\longrightarrow0,
 \label{tetra_L4}\\
 \int_\cZ\E[|D_z(F_m-F)|^4]\mu(\dd z)
 &\longrightarrow0.
 \label{tetra_D4}
\end{align}
The partition cells of the kernels $f_m$ are chosen to have measure at most one.
\end{lemma}

\begin{proof}
We first approximate $f$ by a finite step kernel and then remove the
cells on which two coordinates fall in the same partition atom.  By
the marking construction preceding the lemma, we may assume that
the control measure is nonatomic without changing either
multiple-integral law or the fourth add-one energy.  By
Proposition~\ref{prop_core} and \eqref{core_graph}, there are
bounded symmetric step kernels $g_m$ of finite-measure support such
that
\noi
\begin{align*}
 \|I_q^\eta(g_m)-F\|_{L^4(\Omega)}
 +\|D\{I_q^\eta(g_m)-F\}\|_{L^4(\PP\ot\mu)}
 \longrightarrow0.
\end{align*}
The isometry \eqref{iso} then implies that
$g_m\to f$ in $L^2(\mu^q)$.

By the finite-count construction in \eqref{count_proj}, for each
$m$ there is a finite partition
$\mathcal P_m=\{B_1,\ldots,B_J\}$ of a finite-measure set such that
$g_m$ is constant on the corresponding product cells.  Since the control
measure is nonatomic, for each $k$ we may refine every $B_j$ into
finitely many measurable sets so that the resulting partition
$\mathcal P_{m,k}$ satisfies
$\delta_{m,k}:=\max\{\mu(A):A\in\mathcal P_{m,k}\}\to0$
as $k\to\infty$.
Define
\[
 C_{m,k}=
 \bigcup_{1\leq r<s\leq q}
 \bigcup_{A\in\mathcal P_{m,k}}
 \{(z_1,\ldots,z_q):z_r,z_s\in A\}
 \AND
 g_{m,k}^{\circ}=g_m\ind_{C_{m,k}^c}.
\]
Thus $g_{m,k}^{\circ}$ is tetrahedral relative to
$\mathcal P_{m,k}$.  Let $S_m\in\sZ$ have finite measure and satisfy
$\supp(g_m)\subset S_m^q$.  Then
\[
 \mu^q(C_{m,k}\cap S_m^q)
 \leq
 \binom q2
 \delta_{m,k}\mu(S_m)^{q-1}
 \longrightarrow0.
\]

Let $h_{m,k}=g_{m,k}^{\circ}-g_m$.  Since $g_m$ is bounded and
$\mu(S_m)<\infty$, we have $h_{m,k}\to0$ in every
$L^p(\mu^q)$, $1\leq p<\infty$.  The term with $r=\ell=0$ satisfies
$\|h_{m,k}\otimes h_{m,k}\|_2=\|h_{m,k}\|_2^2\to0$.
For $1\leq r\leq q$, the fully integrated contractions satisfy
$\|h_{m,k}\otimes_r h_{m,k}\|_2\leq\|h_{m,k}\|_2^2\to0$.
If $0\leq\ell<r\leq q$ and $t=r-\ell$, then we deduce from
\eqref{star_section} that
\noi
\begin{align*}
 \|h_{m,k}\star_r^\ell h_{m,k}\|_2^2
 &\leq\sup_{\boldsymbol z\in S_m^t}
 \|(h_{m,k})_{\boldsymbol z}\|_2^2
 \int_{S_m^t}\|(h_{m,k})_{\boldsymbol z}\|_2^2
 \,\mu^t(\dd\boldsymbol z)\longrightarrow0.
\end{align*}
Indeed, if $M_m=2\|g_m\|_\infty$, then the supremum is bounded by
$M_m^2\mu(S_m)^{q-t}$, while the integral equals
$\|h_{m,k}\|_2^2$.  The product formula \eqref{Poi_prod_full} and the
isometry \eqref{iso} now imply
$\|I_q^\eta(h_{m,k})\|_4\to0$.  We then deduce from
\eqref{D4_prelim} that
$\Jfour(I_q^\eta(h_{m,k}))\leq(4q-3)
\cum_4(I_q^\eta(h_{m,k}))\to0$.

On the Gaussian space,
$g_{m,k}^{\circ}\to g_m$ in $L^2(\mu^q)$.  We deduce from
\eqref{Gau_hyper} that
\noi
\begin{align*}
 \big\|I_q^W(g_{m,k}^{\circ})-I_q^W(g_m)\big\|_{L^4(\Omega)}\longrightarrow0.
\end{align*}
Together with $g_m\to f$ in $L^2(\mu^q)$, the same argument gives
$I_q^W(g_m)\to G$ in $L^4(\Omega)$.

For each $m$, choose $k(m)$ so large that the Poisson $L^4$
error, the derivative $L^4$ error, and the Gaussian $L^4$ error caused
by replacing $g_m$ with $g_{m,k(m)}^{\circ}$ are at most $m^{-1}$,
and also $\delta_{m,k(m)}\leq1$.  Then
$f_m=g_{m,k(m)}^{\circ}$ satisfies
\eqref{tetra_L4}--\eqref{tetra_D4}, and every partition cell has
measure at most one.
\qedhere

\end{proof}

The next elementary estimate controls the fourth moments appearing
under the hybrid laws in a Lindeberg replacement.

\begin{lemma}
\label{lem_hyb_4}
Let $R$ be a multilinear polynomial of degree at most $r$ in mutually
independent centered coordinates, each of which is either Gaussian or
centered Poisson.  On a product extension, replace every Gaussian
coordinate by a centered Poisson coordinate with the same variance,
independent of all original coordinates and of the other replacement
coordinates.  Denote the resulting polynomial by $R^\eta$.  Then
\noi
\begin{align}
 \E[|R|^4]\leq3^{2r}\E[|R^{\eta}|^4].
 \label{hyb_4}
\end{align}
\end{lemma}

\begin{proof}
Let $\mathcal P$ be the $\s$-algebra generated by the original
Poisson coordinates.  Conditionally on $\mathcal P$, the random
variable $R$ is a polynomial of degree at most $r$ in independent
Gaussian coordinates.  We deduce from Gaussian hypercontractivity
\eqref{Gau_hyper} that
$\E[R^4 | \mathcal P]\leq3^{2r}
\bigl(\E[R^2 | \mathcal P]\bigr)^2$.
By multilinearity and independence, the conditional second moment
depends on the remaining coordinates only through their variances.
Hence, after taking conditional expectation over the replacement
coordinates,
$\E[R^2 | \mathcal P]=\E[(R^\eta)^2 | \mathcal P]$.
Conditional Jensen's inequality gives
$\bigl(\E[(R^\eta)^2 | \mathcal P]\bigr)^2\leq
\E[(R^\eta)^4 | \mathcal P]$.  Taking expectations proves
\eqref{hyb_4}.
\end{proof}

We are now ready to prove Theorem~\ref{thm4}.

\begin{proof}[Proof of Theorem~\ref{thm4}]

We first prove \eqref{thm5_1}.  By Lemma~\ref{lem_tetra_core},
\eqref{tetra_L4}, \eqref{tetra_D4}, and the reverse triangle
inequality for $\Jfour^{1/4}$ used in \eqref{core_J4}, it suffices to
consider a tetrahedral step kernel relative to pairwise disjoint cells
$A_1,\ldots,A_N$ satisfying $\lambda_i:=\mu(A_i)\leq1$.
Set

\noi
\begin{align*}
 X_i:=\eta(A_i)-\lambda_i
 \AND
 Y_i:=W(A_i).
\end{align*}
Then the two families $(X_i)_{i=1}^N$ and $(Y_i)_{i=1}^N$ are
mutually independent, with $\E[X_i]=\E[Y_i]=0$ and
$\E[X_i^2]=\E[Y_i^2]=\lambda_i$.
Since the kernel is tetrahedral, there is a homogeneous multilinear
polynomial $Q$ of degree $q$ such that
$ F=Q(X_1,\ldots,X_N)$ and $G=Q(Y_1,\ldots,Y_N)$.

Next, we use a Lindeberg replacement argument.  For $i=0,\ldots,N$, define
the hybrid vector
\noi
\begin{align*}
 Z^{(i)}=(Y_1,\ldots,Y_i,X_{i+1},\ldots,X_N).
\end{align*}
Thus, $Z^{(0)}=(X_1,\ldots,X_N)$ and
$Z^{(N)}=(Y_1,\ldots,Y_N)$, and therefore
\noi
\begin{align}
 \E[h(F)]-\E[h(G)]
 =
 \sum_{i=1}^N
 \left\{
 \E[h(Q(Z^{(i-1)}))]
 -
 \E[h(Q(Z^{(i)}))]
 \right\}.
 \label{Lind_tel}
\end{align}

Since $Q$ is multilinear, for each $i$ we may write
\noi
\begin{align*}
 Q(x_1,\ldots,x_N)
 =
 U_i(x_1,\ldots,\widehat{x_i},\ldots,x_N)
 +
 x_iV_i(x_1,\ldots,\widehat{x_i},\ldots,x_N),
\end{align*}
where $U_i$ and $V_i$ do not depend on the $i$th coordinate.
Evaluate $U_i$ and $V_i$ at
\noi
\begin{align*}
 (Y_1,\ldots,Y_{i-1},X_{i+1},\ldots,X_N)
\end{align*}
and denote the resulting random variables by
$U_i^{\rm hyb}$ and $V_i^{\rm hyb}$.  We may therefore write

\noi
\begin{align*}
 Q(Z^{(i-1)})
 =
 U_i^{\rm hyb}+X_iV_i^{\rm hyb}
 \AND
 Q(Z^{(i)})
 =
 U_i^{\rm hyb}+Y_iV_i^{\rm hyb}.
\end{align*}
Moreover, $(U_i^{\rm hyb},V_i^{\rm hyb})$ is independent of both
$X_i$ and $Y_i$.  Taylor's formula gives, for every $x\in\R$,

\noi
\begin{align*}
 h(U_i^{\rm hyb}+xV_i^{\rm hyb})
 &=
 h(U_i^{\rm hyb})
 +h'(U_i^{\rm hyb})xV_i^{\rm hyb}
 +\frac12h''(U_i^{\rm hyb})x^2
 (V_i^{\rm hyb})^2
 +R_i(x),
\end{align*}
where
$|R_i(x)|\leq\frac16\|h^{(3)}\|_\infty
|x|^3|V_i^{\rm hyb}|^3$.
Since $X_i$ and $Y_i$ have matching first two moments, the constant,
linear, and quadratic terms cancel after taking expectations
so that

\noi
\begin{align}
 &\left|
 \E h(U_i^{\rm hyb}+X_iV_i^{\rm hyb})
 -
 \E h(U_i^{\rm hyb}+Y_iV_i^{\rm hyb})
 \right|
\leq
 \frac{1}{6} \| h^{(3)}\|_\infty
 \E[|V_i^{\rm hyb}|^3]
 \left(
 \E|X_i|^3+\E|Y_i|^3
 \right).
 \label{Lind_step}
\end{align}

A direct computation, using $\lambda_i\leq1$, gives

\noi
 \begin{align} \label{coord_third}
\max\{\E[|X_i|^3],\E[|Y_i|^3]\}\leq2\lambda_i.
\end{align}

Let $V_i^\eta$ be the polynomial $V_i$ evaluated at the all-Poisson
vector obtained from $(X_1,\ldots,X_N)$ by deleting the $i$th
coordinate.  Since $V_i$ has degree at most $q-1$, Lemma~\ref{lem_hyb_4}
and equality of the coordinate variances give
\noi
\begin{align}
 \E[(V_i^{\rm hyb})^4]
 &\leq3^{2(q-1)}\E[(V_i^\eta)^4],
 \quad
 \E[(V_i^{\rm hyb})^2]=\E[(V_i^\eta)^2].
 \label{Vi_hybrid}
\end{align}

Since $F=Q(X_1,\ldots,X_N)$ is multilinear,
\noi
\begin{align*}
 D_zF=\sum_{i=1}^N\ind_{A_i}(z)V_i^\eta.
\end{align*}
We deduce from \eqref{D_L2} with $r=1$ and from the definition
\eqref{J4F}, respectively, that
\noi
\begin{align}
 \sum_{i=1}^N\lambda_i\E[(V_i^\eta)^2]
 =q\sigma^2
 \AND
 \sum_{i=1}^N\lambda_i\E[(V_i^\eta)^4]
 =\Jfour(F).
 \label{infl_fourth}
\end{align}

By Cauchy--Schwarz, \eqref{Vi_hybrid}, and \eqref{infl_fourth},
we obtain

\noi
\begin{align*}
 \sum_{i=1}^N
 \lambda_i\E[|V_i^{\rm hyb}|^3]
 &\leq
 \left(
 \sum_{i=1}^N
 \lambda_i\E[(V_i^{\rm hyb})^2]
 \right)^{1/2}
 \left(
 \sum_{i=1}^N
 \lambda_i\E[(V_i^{\rm hyb})^4]
 \right)^{1/2}\\
 &\leq
 3^{q-1}
 \sqrt{q\sigma^2\Jfour(F)}.
\end{align*}
Combining this with
\eqref{Lind_tel}, \eqref{Lind_step}, and
\eqref{coord_third}, we obtain
\noi
\begin{align*}
 |\E[h(F)]-\E[h(G)]|
 &\leq
 3^{q-1}\sqrt q\,
 \|h^{(3)}\|_\infty
 \sqrt{\sigma^2\Jfour(F)}.
\end{align*}
This proves \eqref{thm5_1} with $K_q=3^{q-1}\sqrt q$.

\medskip

We next prove \eqref{thm5_2}.  Let $N\sim\NN(0,1)$ and, for a
one-Lipschitz function $h$, define
$h_\eps(x)=\E[h(x+\eps N)]$.  Differentiating under the expectation
and using $\E|N|\leq1$ and $\E|N^2-1|\leq2$, we obtain
$\|h-h_\eps\|_\infty\leq\eps$ and
$\|h_\eps^{(3)}\|_\infty\leq2\eps^{-2}$.  Applying
\eqref{thm5_1} to $h_\eps$, with
$A=\sqrt{\sigma^2\Jfour(F)}$, gives
\[
 |\E[h(F)]-\E[h(G)]|
 \leq2\eps+2K_qA\eps^{-2}.
\]
The conclusion is immediate when $A=0$.  Otherwise, the choice
$\eps=(K_qA)^{1/3}$ yields
$|\E[h(F)]-\E[h(G)]|\leq4K_q^{1/3}A^{1/3}$.  Taking the supremum
over all one-Lipschitz $h$ proves \eqref{thm5_2} with
$L_q=4K_q^{1/3}$.
\qedhere

\end{proof}

\subsection{Why the smooth-test order does not extend to Wasserstein distance}
\label{SEC_50b}
\begin{proof}[Proof of Proposition~\ref{prop17}]
We first prove part~\textup{(a)}.  Put
$a=\lambda^{1/q}$ and choose pairwise disjoint sets
$A_1,\ldots,A_q$ of measure $a$.  Let
\noi
\begin{align*}
 F_\lambda
 =\prod_{k=1}^q\{\eta(A_k)-a\}
 \AND
 G_\lambda
 =\prod_{k=1}^qW(A_k).
\end{align*}
By \eqref{disjoint_prod}, these are the Poisson and Gaussian multiple
integrals of the same symmetric kernel.
The second-moment identity \eqref{prop17_1} follows from independence.
Moreover,
\noi
\begin{align*}
 D_zF_\lambda
 =\sum_{k=1}^q\ind_{A_k}(z) \prod_{j\neq k}\{\eta(A_j)-a\}.
\end{align*}
We deduce \eqref{prop17_2} from the fact that
the fourth centered moment of $\Pois(a)$ is
$a+3a^2$.

For every $h\in\cH_2$, we have
$|h(x)-h(0)-h'(0)x|\leq x^2/2$.
Since $F_\lambda$ and $G_\lambda$ are centered
with common
variance $\lambda$,

\noi
\begin{align}\label{d2_eq_u}
 d_2(F_\lambda,G_\lambda)\leq\lambda.
\end{align}
For the reverse bound, let
\noi
\begin{align*}
 h_0(x)=
 \begin{cases}
 0,&x\leq0,\\
 x^2/2,&0<x<1,\\
 x-1/2,&x\geq1.
 \end{cases}
\end{align*}
Then $h_0\in\cH_2$.
With $a= \lambda^{1/q}$ and

\noi
\begin{align*}
 B_a=\bigcap_{k=1}^q\{\eta(A_k)\geq1\},
\end{align*}
independence gives
\noi
\begin{align*}
 \PP(B_a)&=(1-e^{-a})^q, \quad
 \E[F_\lambda\ind_{B_a}]=(ae^{-a})^q, \quad
 \E[F_\lambda^2\ind_{B_a}]=(a-a^2e^{-a})^q.
\end{align*}
Since $a^q=\lambda$, these identities yield
$\PP(B_a)=\lambda+O_q(\lambda a)$,
$\E[F_\lambda^2\ind_{B_a^c}]=O_q(\lambda a)$, and
$\E[(F_\lambda-1)^2\ind_{B_a}]=O_q(\lambda a)$.  In particular,
Cauchy--Schwarz gives
$\E[|F_\lambda-1|\ind_{B_a}]=o(\lambda)$.
Since $h_0$ is one-Lipschitz and $h_0(1)=1/2$,
$h_0(F_\lambda)-1/2\geq-|F_\lambda-1|$, and hence

\noi
\begin{align*}
 \E[h_0(F_\lambda)]
 &\geq\E[h_0(F_\lambda)\ind_{B_a}]
 \geq\frac12\PP(B_a)-\E[|F_\lambda-1|\ind_{B_a}] \\
& \geq\frac{\lambda}{2}-o(\lambda).
\end{align*}
On the other hand,
\[
 G_\lambda=\sqrt{\lambda}\,V_q,
\]
where $V_q$ is the product of $q$ independent standard normal
variables.  By the definition of $h_0$,
\[
 0\leq
 \frac{x_+^2}{2}-h_0(x)
 =
 \frac{(x-1)^2}{2}\ind_{\{x\geq1\}}
 \leq
 \frac{x^2}{2}\ind_{\{x\geq1\}}.
\]
\[
\begin{aligned}
 0
 \leq
 \frac12\E[(G_\lambda)_+^2]
 -\E[h_0(G_\lambda)]
 &\leq \frac12 \E[G_\lambda^2\ind_{\{G_\lambda\geq1\}}]\\
 &=
 \frac{\lambda}{2}
 \E\left[
 V_q^2
 \ind_{\{V_q\geq\lambda^{-1/2}\}}
 \right]
 =
 o(\lambda),
\end{aligned}
\]
where the last step follows from $V_q^2\in L^1(\Omega)$ and dominated
convergence.  Since $V_q$ is symmetric,
\[
 \E[(G_\lambda)_+^2]
 =
 \frac12\E[G_\lambda^2]
 =
 \frac{\lambda}{2}.
\]
Therefore
\[
 \E[h_0(G_\lambda)]
 =
 \frac14\E[G_\lambda^2]+o(\lambda)
 =
 \frac{\lambda}{4}+o(\lambda).
\]

Therefore,
$d_2(F_\lambda,G_\lambda)\geq\frac{\lambda}{5}$
for all sufficiently small $\lambda$,
which together with \eqref{d2_eq_u} proves
\eqref{prop17_3}.

Finally, the one-Lipschitz test $h(x)=|x|$ gives

\noi
\begin{align*}
 \dW(F_\lambda,G_\lambda)
 \geq\E[|G_\lambda|]-\E[|F_\lambda|]
 \geq
 \left(\frac2\pi\right)^{q/2}\sqrt{\lambda}
 -2^q\lambda e^{-qa},
\end{align*}
where we used $\E|\eta(A_k)-a|=2ae^{-a}$
for $0<a<1$.
The reverse estimate
$ \dW(F_\lambda,G_\lambda)
 \leq\E|F_\lambda|+\E|G_\lambda|$
is of order $\sqrt{\lambda}$.  This proves
\eqref{prop17_4} and thus completes the proof of part~\textup{(a)}.

\medskip

We now prove part~\textup{(b)}.  We combine a diffuse approximation
of $N^2-1$ with a small second-chaos perturbation.

Choose $B_m$ with $\mu(B_m)=m$ and set
\noi
\begin{align*}
 g_m=m^{-1/2}\ind_{B_m},
 \quad
 \bar F_m=I_2^\eta(g_m^{\ot2}),
 \AND
 \bar G_m=I_2^W(g_m^{\ot2}).
\end{align*}
If
\[
 X_m=\frac{\eta(B_m)-m}{\sqrt m},
\]
then we deduce from the product formula that

\noi
\begin{align*}
 \bar F_m
 =X_m^2-\frac{X_m}{\sqrt m}-1
\AND
 \bar G_m\stackrel{\mathrm d}=Y:=N^2-1,
\end{align*}
where $N\sim\NN(0,1)$.  Moreover,
$\E[\bar F_m^2]=\E[\bar G_m^2]=2$, and
$\bar F_m\xrightarrow{\mathrm d}Y$ as $m\to\infty$.

Since the second moments are uniformly bounded, the first absolute
moments are uniformly integrable.  The convergence criterion in
Section~\ref{SEC_dist} therefore gives
\noi
\begin{align}
 \dW(\bar F_m,Y)\longrightarrow0.
 \label{metric_base_W}
\end{align}
Moreover,
with
$D_z\bar F_m=2g_m(z)X_m$,
we have
\noi
\begin{align}
 \Jfour(\bar F_m)
 =
 \frac{16}{m}\left(3+\frac1m\right)
 \longrightarrow0.
 \label{metric_base_J4}
\end{align}

We next construct a fixed perturbation.  Fix $0<a<1$, choose
disjoint sets $A_1,A_2$, and choose the sets $B_m$ above disjoint from
$A_1\cup A_2$, with
$\mu(A_1)=\mu(A_2)=a$.  Put
\noi
\begin{align*}
 U_a
 =
 \frac{(\eta(A_1)-a)(\eta(A_2)-a)}{a}
 \AND
 V_a
 =
 \frac{W(A_1)W(A_2)}{a}.
\end{align*}
They are Poisson and Gaussian double integrals of the same kernel
$k_a =
 \frac1a\sym(\ind_{A_1}\ot\ind_{A_2})
 $,
such that $\E[U_a^2]=\E[V_a^2]=1$
and
$\Jfour(U_a)=2(1+3a)/a^2$.

We choose $a$ so that $U_a$ and $V_a$ can be distinguished near the
left endpoint $-1$ of the support of $Y$.  Since $0<a<1$, $U_a<0$
precisely when exactly one of the two Poisson counts vanishes, and
\noi
\begin{align*}
 \E[(-U_a)_+^{3/2}]
 =
 2e^{-a}
 \E\left[
 (\eta(A_1)-a)^{3/2}
 \ind_{\{\eta(A_1)\geq1\}}
 \right]
 \longrightarrow0
\end{align*}
as $a\downarrow0$.  On the other hand,
$\E[(-V_a)_+^{3/2}]
=\frac12(\E|N|^{3/2})^2>0$.
We may therefore fix $a>0$ sufficiently small that
\noi
\begin{align}
 \E[(-U_a)_+^{3/2}]
 \neq
 \E[(-V_a)_+^{3/2}].
 \label{metric_moment_sep}
\end{align}

To exploit this difference, take the one-Lipschitz function
$ h(x)=(-1-x)_+$
and define, for $t\geq0$,
$
 \Psi(t)=\E[(t-N^2)_+].
$
Writing $\phi(x)=(2\pi)^{-1/2}e^{-x^2/2}$, we have
\noi
\begin{align*}
 \Psi(t)=2\int_0^{\sqrt t}(t-x^2)\phi(x)\,\dd x
 =\frac{4}{3\sqrt{2\pi}}t^{3/2}+o(t^{3/2}),
 \quad t\downarrow0.
\end{align*}
Moreover, $\Psi(t)\leq Ct^{3/2}$ for all $t\geq0$: use the
boundedness of $\phi$ for $0\leq t\leq1$ and the bound
$\Psi(t)\leq t$ for $t\geq1$.  If $R$ is independent of $N$ and
$\E|R|^{3/2}<\infty$, then conditioning on $R$ gives
\noi
\begin{align*}
 \E[h(Y+\eps R)]
 =\E\bigl[\Psi\bigl(\eps(-R)_+\bigr)\bigr].
\end{align*}
Hence dominated convergence yields
\noi
\begin{align*}
 \eps^{-3/2}\E[h(Y+\eps R)]
 \xrightarrow{\eps\to0}
 \frac{4}{3\sqrt{2\pi}}
 \E[(-R)_+^{3/2}].
\end{align*}
Applying this to $R=U_a$ and $R=V_a$, and using
\eqref{metric_moment_sep}, we obtain some $c_a>0$ such that
\noi
\begin{align*}
 \dW(Y+\eps U_a,Y+\eps V_a)
 \geq c_a\eps^{3/2}
\end{align*}
for all sufficiently small $\eps$.

Choose $m=m(\eps)\to\infty$ sufficiently fast that
\noi
\begin{align*}
 \dW(\bar F_{m(\eps)},Y)
 =o(\eps^{3/2})
 \AND
 \Jfour(\bar F_{m(\eps)})
 =o(\eps^4),
\end{align*}
which is possible by
\eqref{metric_base_W}--\eqref{metric_base_J4}.  Define
\noi
\begin{align*}
 \widetilde F_\eps
 =
 \bar F_{m(\eps)}+\eps U_a
 \AND
 \widetilde G_\eps
 =
 \bar G_{m(\eps)}+\eps V_a.
\end{align*}
Since the two components have disjoint supports,
$\widetilde F_\eps$ and $\widetilde G_\eps$ are same-kernel double
integrals, with common kernel
$
 g_{m(\eps)}^{\ot2}+\eps k_a.
$
Moreover,
$
 \E[\widetilde F_\eps^2]
 =\E[\widetilde G_\eps^2]
 =2+\eps^2
$
and
$
 \Jfour(\widetilde F_\eps)
 =
 \Jfour(\bar F_{m(\eps)})
 +\eps^4\Jfour(U_a)
 \asymp\eps^4.
$

By the triangle inequality and the convolution contraction
\eqref{dist_conv},
\noi
\begin{align*}
 \dW(\widetilde F_\eps,\widetilde G_\eps)
 &\geq
 \dW(Y+\eps U_a,Y+\eps V_a)
 -\dW(\bar F_{m(\eps)}+\eps U_a,Y+\eps U_a)\\
 &\geq
 \dW(Y+\eps U_a,Y+\eps V_a)
 -\dW(\bar F_{m(\eps)},Y)
 \gtrsim\eps^{3/2}.
\end{align*}

Finally, set
\noi
\begin{align*}
 r_\eps
 =
 \sqrt{\frac{2}{2+\eps^2}},
 \quad
 F_\eps=r_\eps\widetilde F_\eps
 \AND
 G_\eps=r_\eps\widetilde G_\eps.
\end{align*}
Then $F_\eps$ and $G_\eps$ are still same-kernel double integrals.
Using the homogeneity of $\dW$ and $\Jfour$, and the fact that
$r_\eps\to1$, we obtain
\noi
\begin{align*}
 \E[F_\eps^2]=\E[G_\eps^2]=2,
 \quad
 \Jfour(F_\eps)\asymp\eps^4,
 \AND
 \dW(F_\eps,G_\eps)\gtrsim\eps^{3/2}.
\end{align*}
For any sequence $\eps_n\downarrow0$, set
$F_n=F_{\eps_n}$ and $G_n=G_{\eps_n}$.  Then
\noi
\begin{align*}
 \frac{\dW(F_n,G_n)}
 {\sqrt{\E[F_n^2]\Jfour(F_n)}}
 \gtrsim
 \eps_n^{-1/2}
 \longrightarrow\infty,
\end{align*}
which proves part~\textup{(b)}.

This completes the proof of Proposition~\ref{prop17}.
\end{proof}

For the moment-transfer and parity arguments, the product formulas
\eqref{Poi_prod_full} and \eqref{Gau_prod_full} show that the Gaussian
square contains only the fully integrated contractions
$f\star_r^r f$, whereas the Poisson square also contains
$f\star_r^\ell f$ with $\ell<r$.  The next two subsections show that
$\Jfour(F)$ controls every such Poisson-specific contraction.

\subsection{Propagation of the fourth add-one energy}
\label{SEC_51}

For $1\leq t\leq q$, define
\noi
\begin{align*}
 \mathcal J_{4,t}(F)
 =\int_{\cZ^t}
 \E\left[\left|D^t_{z_1,\ldots,z_t}F\right|^4\right]
 \mu^t(\dd z_1\cdots\dd z_t).
\end{align*}
Thus $\mathcal J_{4,1}(F)=\Jfour(F)$.  Put
\noi
\begin{align*}
 A_{q,1}=1,
 \quad
 A_{q,t}=\prod_{j=1}^{t-1}\bigl(4(q-j)-3\bigr),
 \AND 2\leq t\leq q.
\end{align*}

\begin{lemma}
\label{lem_J4_prop}
For every $1\leq t\leq q$,

\noi
\begin{align}
 \mathcal J_{4,t}(F)
 \leq
 A_{q,t}\Jfour(F). \label{J4_prop}
\end{align}

\end{lemma}

\begin{proof}
This is the recursive derivative estimate proved in
\cite[Proof of Theorem~1.8]{Zhe26a}.  Indeed,
we have
$ \mathcal J_{4,t}(F)
 \leq
 \bigl(4(q-t+1)-3\bigr)
 \mathcal J_{4,t-1}(F)$
 for $2\leq t\leq q$.
Iterating this inequality and using
$\mathcal J_{4,1}(F)=\Jfour(F)$ yields \eqref{J4_prop}.
\end{proof}

\subsection{Control of the Poisson-specific contractions}
\label{SEC_52}

The following elementary consequence of the iterated-derivative
estimate will be used below.  It reverses, at the level needed here, the
usual implication from vanishing Poisson contractions to vanishing
fourth add-one energy; see \cite[Remark~1.8(b)]{DP18b} and
\cite[Lemma~4.1]{PT13}.

\begin{lemma}
\label{lem_offdiag}
For every $0\leq\ell<r\leq q$,
\noi
\begin{align}
 \big\| f\star_r^\ell f \big\|_{L^2(\mu^{2q-r-\ell})}
 \leq
 \frac{(q-r+\ell)!}{(q!)^2}
 \sqrt{A_{q,r-\ell}}\sqrt{\Jfour(F)}.
 \label{offdiag_J4}
\end{align}
\end{lemma}

\begin{proof}
Fix $0\leq\ell<r\leq q$ and put $t=r-\ell\geq1$.  We deduce from
\eqref{star_D} and \eqref{J4_prop} that
\[
 \|f\star_r^\ell f\|_2
 \leq\frac{(q-t)!}{(q!)^2}\mathcal J_{4,t}(F)^{1/2}
 \leq\frac{(q-t)!}{(q!)^2}
 \sqrt{A_{q,t}}\sqrt{\Jfour(F)}.
\]
This proves \eqref{offdiag_J4}.
\end{proof}

\begin{remark}\rm
\label{rem_offdiag_vanish}
For a sequence $F_n=I_q^\eta(f_n)$ with $f_n$ symmetric,
we have the following implication:

\noi
\begin{align*}
 \Jfour(F_n)\longrightarrow0
 \quad\Longrightarrow\quad
 \big\|f_n\star_r^\ell f_n\big\|_{L^2(\mu^{2q-r-\ell})}
 \longrightarrow0,
 \quad 0\leq\ell<r\leq q.
\end{align*}

Thus, the add-one Lindeberg condition forces every contraction term with
no analogue in the Gaussian product formula \eqref{Gau_prod_full} to vanish.
\end{remark}

\subsection{Proof of Proposition~\ref{prop18}: third moments}
\label{SEC_53}
Define
\noi
\begin{align}
 B_q
 =\frac1{(q!)^{3/2}}
 \sum_{r=\lfloor q/2\rfloor+1}^{q}
 r!\binom qr^2\binom r{q-r}
 (2q-2r)!
 \sqrt{A_{q,2r-q}}.
 \label{Bq_def}
\end{align}

\begin{proof}[Proof of \eqref{prop18_1}]
We deduce from Lemma~\ref{lem_offdiag} that every contraction
$f\star_r^\ell f$ with $\ell<r$ is square integrable.  The fully
integrated contractions $f\star_r^r f=f\otimes_r f$ are square
integrable directly by the Hilbert--Schmidt contraction inequality
$\|f\otimes_r f\|_2\leq\|f\|_2^2$.  Hence all terms required by
the product formula \eqref{Poi_prod_full} are in $L^2$, and we deduce
from that formula that
\noi
\begin{align}
 F^2
 =\sum_{r=0}^q r!\binom qr^2
 \sum_{\ell=0}^r\binom r\ell
 I_{2q-r-\ell}^\eta
 \left(f \wt{\star_r^\ell} f\right).
 \label{F2_Poi_full}
\end{align}
Now
$\E[F^3]=\inner{F}{F^2}_{L^2(\PP)}$.  Since
$F\in\cC_q^\eta$, a term in \eqref{F2_Poi_full} contributes, by
orthogonality of distinct Poisson chaoses, if and only if
\noi
\begin{align*}
 2q-r-\ell=q,
 \quad\text{equivalently}\quad r+\ell=q.
\end{align*}
Thus $\ell=q-r$, which is admissible exactly when
$r\geq\lceil q/2\rceil$.  Applying the order-$q$ isometry and using
that $f$ is symmetric,
\noi
\begin{align*}
 \E[F^3]
 =q!\sum_{r=\lceil q/2\rceil}^q
 r!\binom qr^2\binom r{q-r}
 \inner{f\star_r^{q-r}f}{f}_{L^2(\mu^q)}.
\end{align*}

On the Gaussian space, we deduce from \eqref{Gau_prod_full} that
\noi
\begin{align}
 G^2
 =\sum_{r=0}^q r!\binom qr^2
 I_{2q-2r}^W
 \left( f \wt{\star_r^r} f\right).
 \label{G2_Gau_full}
\end{align}
A term contributes to $\E[G^3]$ exactly when
$2q-2r=q$.  If $q$ is odd there is no such integer $r$, so
$\E[G^3]=0$.  If $q$ is even, the unique contribution is
$r=q/2$, and
\noi
\begin{align*}
 \E[G^3]
 =q!\left(\frac q2\right)!
 \binom q{q/2}^2
 \inner{f\star_{q/2}^{q/2}f}{f}_{L^2(\mu^q)}.
\end{align*}
Hence in either parity
\noi
\begin{align}
 \E[F^3]-\E[G^3]
 =q!\sum_{r=\lfloor q/2\rfloor+1}^q
 r!\binom qr^2\binom r{q-r}
 \inner{f\star_r^{q-r}f}{f}.
 \label{third_diff_exact}
\end{align}

For every index in this sum, put $\ell=q-r$ and
$t=r-\ell=2r-q\geq1$.  Applying \eqref{star_D} and then
\eqref{J4_prop}, we obtain
\noi
\begin{align}
 \|f\star_r^{q-r}f\|_2
 \leq
 \frac{(2q-2r)!}{(q!)^2}
 \sqrt{A_{q,2r-q}}\sqrt{\Jfour(F)}.
 \label{relevant_offdiag}
\end{align}

Applying Cauchy--Schwarz to each inner product in
\eqref{third_diff_exact}, then inserting
\eqref{relevant_offdiag} and
$\|f\|_2=\sqrt{\E[F^2]/q!}$, gives
\begin{align*}
 |\E[F^3]-\E[G^3]|
 &\leq
 \frac{\sqrt{\E[F^2]\Jfour(F)}}{(q!)^{3/2}}
 \sum_{r=\lfloor q/2\rfloor+1}^q
 r!\binom qr^2\binom r{q-r}(2q-2r)!
 \sqrt{A_{q,2r-q}}.
\end{align*}
Under the normalization
$\E[F^2]=2\nu$, this is exactly
\eqref{prop18_1}.
\end{proof}
\subsection{Proof of Proposition~\ref{prop18}: fourth moments}
\label{SEC_54}

We now compare $F^2$ with the Poisson chaos expansion
obtained from the Gaussian square $G^2$ by replacing each Gaussian
multiple integral with the Poisson multiple integral having the same
kernel.  More precisely, by
\eqref{G2_Gau_full} and \eqref{chaos_transport}, define
\noi
\begin{align}
 \cG_q(f)
 =\mathsf T(G^2)
 =\sum_{r=0}^q r!\binom qr^2
 I_{2q-2r}^\eta
 \left( f\wt{\star_r^r} f \right).
 \label{Gaussian_skeleton}
\end{align}
Set
\noi
\begin{align}
 C_q
 =\frac1{(q!)^2}
 \sum_{r=1}^q\sum_{\ell=0}^{r-1}
 r!\binom qr^2\binom r\ell
 \sqrt{(2q-r-\ell)!}(q-r+\ell)!
 \sqrt{A_{q,r-\ell}}.
 \label{Cq_def}
\end{align}

\begin{proposition}
\label{prop_square_gauss}
For every $F=I_q^\eta(f)\in\cC_q^\eta\cap L^4(\Omega)$,
\noi
\begin{align}
 \|F^2-\cG_q(f)\|_{L^2(\PP)}
 \leq C_q\sqrt{\Jfour(F)}.
 \label{square_gauss}
\end{align}
\end{proposition}

\begin{proof}
Subtracting the $\ell=r$ terms in \eqref{F2_Poi_full}, we obtain the
exact identity
\noi
\begin{align*}
 F^2-\cG_q(f)
 =\sum_{r=1}^q r!\binom qr^2
 \sum_{\ell=0}^{r-1}\binom r\ell
 I_{2q-r-\ell}^\eta
 \left( f\wt{\star_r^\ell} f \right).
\end{align*}
For $k=2q-r-\ell$, the multiple-integral isometry and contractivity
of symmetrization give
\noi
\begin{align}
 \|I_k^\eta( f\wt{\star_r^\ell} f )\|_2
 &\leq\sqrt{k!}\|f\star_r^\ell f\|_2.
 \label{sym_contract}
\end{align}

The triangle inequality, \eqref{sym_contract}, and
\eqref{offdiag_J4} therefore give
\begin{align*}
 \|F^2-\cG_q(f)\|_2
 &\leq
 \sum_{r=1}^q\sum_{\ell=0}^{r-1}
 r!\binom qr^2\binom r\ell
 \sqrt{(2q-r-\ell)!}
 \|f\star_r^\ell f\|_2\\
 &\leq C_q\sqrt{\Jfour(F)},
\end{align*}
which is \eqref{square_gauss}.
\end{proof}

\begin{proof}[Proof of \eqref{prop18_2}]
By \eqref{G2_Gau_full}, \eqref{Gaussian_skeleton}, and the Poisson and
Gaussian isometries in \eqref{Poi_Gau_iso}, we have
$\|\cG_q(f)\|_2^2=\|G^2\|_2^2=\E[G^4]$.
Combining this identity with the reverse triangle inequality and
Proposition~\ref{prop_square_gauss}, we obtain
\begin{align*}
 \big|\sqrt{\E[F^4]}-\sqrt{\E[G^4]}\big|
 &=\big|\|F^2\|_2-\|\cG_q(f)\|_2\big|\\
 &\leq\|F^2-\cG_q(f)\|_2
 \leq C_q\sqrt{\Jfour(F)}.
\end{align*}
This proves \eqref{prop18_2} and completes the proof of
Proposition~\ref{prop18}.
\end{proof}

\subsection{Proof of Proposition~\ref{prop19}: the Gamma defect and even orders}
\label{SEC_55}

Recall from \eqref{intro_Del} that for a Gaussian-chaos
variable $G$ with $\E[G^2]=2\nu$ we write
$\Del_\nu(G)$ for the corresponding third--fourth Gamma defect.
The next estimate combines the odd and even moment transfers.

\begin{lemma}
\label{lem_defect_transfer}
If $\E[F^2]=\E[G^2]=2\nu$, then
\noi
\begin{align}
 |\Del_\nu(F)-\Del_\nu(G)|
 &\leq12B_q\sqrt{2\nu\Jfour(F)}+
 C_q\sqrt{\Jfour(F)}
 \left(4\nu\,3^q+C_q\sqrt{\Jfour(F)}\right).
 \label{defect_transfer}
\end{align}
\end{lemma}

\begin{proof}
Gaussian hypercontractivity \eqref{Gau_hyper} and
$\E[G^2]=2\nu$ imply
$\sqrt{\E[G^4]}\leq2\nu\,3^q$.  We deduce from
\eqref{prop18_2} that
$\sqrt{\E[F^4]}\leq2\nu\,3^q+C_q\sqrt{\Jfour(F)}$, and hence
\begin{align*}
 |\E[F^4]-\E[G^4]|
 &\leq C_q\sqrt{\Jfour(F)}
 \left(4\nu\,3^q+C_q\sqrt{\Jfour(F)}\right).
\end{align*}
The variance-dependent constants in the two defects are identical.
Combining the preceding estimate with \eqref{prop18_1}, we obtain
\eqref{defect_transfer}.
\end{proof}

\begin{proof}[Proof of \eqref{prop19_1}--\eqref{prop19_2}]
Under the assumptions of part~{\rm(i)}, the right-hand side of
\eqref{defect_transfer} tends to zero.  This proves
\eqref{prop19_1}.

Assume now that $q$ is even.  The first equivalence below follows
from \eqref{prop19_1}, the second from the Gaussian centered Gamma
theorem \eqref{NP_Gam}, and the third from Theorem~\ref{thm4}:
\[
 \Del_\nu(F_n)\to0
 \quad\Longleftrightarrow\quad
 \Del_\nu(G_n)\to0
 \quad\Longleftrightarrow\quad
 G_n\xrightarrow{\mathrm d}Z_\nu
 \quad\Longleftrightarrow\quad
 F_n\xrightarrow{\mathrm d}Z_\nu.
\]
Here the last step uses $\E[F_n^2]=2\nu$ and
$\Jfour(F_n)\to0$, which imply $\dW(F_n,G_n)\to0$.  This proves
\eqref{prop19_2}.
\end{proof}

\subsection{Proof of Proposition~\ref{prop19}: odd orders}
\label{SEC_56}

When $q$ is odd, the Gaussian square has no chaos-$q$ component.
Therefore, for every $F$ with $\E[F^2]=2\nu$, we deduce from
\eqref{prop18_1} that
\noi
\begin{align}
 |\E[F^3]|
 \leq B_q\sqrt{2\nu\Jfour(F)}.
 \label{odd_third_bound}
\end{align}

\begin{proof}[Proof of Proposition~\ref{prop19}{\rm(ii)}]
Assume that $q$ is odd and $F_n\xrightarrow{\mathrm d}Z_\nu$; recall
that $\E[F_n^2]=2\nu$.  Set
$L=\liminf_{n\to\infty}\Jfour(F_n)$.  There is nothing to prove if
$L=\infty$.  Otherwise, choose $(n_k)$ such that
$\Jfour(F_{n_k})\to L$.  Gaussian hypercontractivity
\eqref{Gau_hyper} and \eqref{prop18_2} yield
\noi
\begin{align*}
 \sup_k\sqrt{\E[F_{n_k}^4]}
 \leq2\nu\,3^q+C_q\sup_k\sqrt{\Jfour(F_{n_k})}<\infty.
\end{align*}
We deduce from \eqref{UI_tail}, with $r=3$ and $s=4$, that
$\{|F_{n_k}|^3:k\geq1\}$ is uniformly integrable.  Hence
$\E[F_{n_k}^3]\to\E[Z_\nu^3]=8\nu$.  We deduce from
\eqref{odd_third_bound} that
$
 8\nu\leq B_q\sqrt{2\nu L},
$
which is equivalent to \eqref{prop19_3}.
\end{proof}

\medskip
\noi
$\bullet$ {\bf Statement on the use of generative AI.}
Generative AI tools were used during the preparation of this
manuscript to assist with language editing, 
 LaTeX formatting, and preliminary literature searches.
The authors independently verified all references, calculations,
proofs, and conclusions and take full responsibility for the contents
of the paper.


\begin{thebibliography}{99}

\bibitem[APY21]{APY21}
E.~Azmoodeh, G.~Peccati, X.~Yang,
\emph{Malliavin--Stein method: a survey of some recent developments},
Modern Stoch. Theory Appl. 8 (2021), no.~2, 141--177.


\bibitem[DP18a]{DP18a}
C.~D\"obler, G.~Peccati,
\emph{The Gamma Stein equation and non-central de Jong theorems},
Bernoulli 24 (2018), 3384--3421.

\bibitem[DP18b]{DP18b}
C.~D\"obler, G.~Peccati,
\emph{The fourth moment theorem on the Poisson space},
Ann. Probab. 46 (2018), 1878--1916.

\bibitem[DVZ18]{DVZ18}
C.~D\"obler, A.~Vidotto, G.~Zheng,
\emph{Fourth moment theorems on the Poisson space in any dimension},
Electron. J. Probab. 23 (2018), paper no.~36, 27 pp.

\bibitem[ET14]{ET14}
P.~Eichelsbacher, C.~Th\"ale,
\emph{New Berry--Esseen bounds for non-linear functionals of Poisson
random measures},
Electron. J. Probab. 19 (2014), paper no.~102, 25 pp.

\bibitem[FT16]{FT16}
T.~Fissler, C.~Th\"ale,
\emph{A four moments theorem for Gamma limits on a Poisson chaos},
ALEA Lat. Am. J. Probab. Math. Stat. 13 (2016), 163--192.

\bibitem[FT17]{FT17}
T.~Fissler, C.~Th\"ale,
\emph{Erratum to: ``A four moments theorem for Gamma limits on a
Poisson chaos''},
ALEA Lat. Am. J. Probab. Math. Stat. 14 (2017), 245--247.

\bibitem[LRP13]{LRP13}
R.~Lachi\`eze-Rey, G.~Peccati,
\emph{Fine Gaussian fluctuations on the Poisson space, I:
contractions, cumulants and geometric random graphs},
Electron. J. Probab. 18 (2013), paper no.~32, 32 pp.


\bibitem[Las16]{Las16}
G.~Last,
\emph{Stochastic analysis for Poisson processes},
in: G.~Peccati, M.~Reitzner (eds.),
\emph{Stochastic Analysis for Poisson Point Processes},
Bocconi \& Springer Series, vol.~7, Springer, Cham, 2016, pp.~1--36.

\bibitem[LP18]{LP18}
G.~Last, M.~Penrose,
\emph{Lectures on the Poisson Process},
Institute of Mathematical Statistics Textbooks, vol.~7,
Cambridge University Press, Cambridge, 2018.


\bibitem[NP09a]{NP09a}
I.~Nourdin, G.~Peccati,
\emph{Stein's method on Wiener chaos},
Probab. Theory Related Fields 145 (2009), 75--118.

\bibitem[NP09b]{NP09b}
I.~Nourdin, G.~Peccati,
\emph{Noncentral convergence of multiple integrals},
Ann. Probab. 37 (2009), 1412--1426.

\bibitem[NP12]{NP12}
I.~Nourdin, G.~Peccati,
\emph{Normal Approximations with Malliavin Calculus:
From Stein's Method to Universality},
Cambridge Tracts in Mathematics, vol.~192,
Cambridge University Press, Cambridge, 2012.

\bibitem[NPR10]{NPR10}
I.~Nourdin, G.~Peccati, G.~Reinert,
\emph{Invariance principles for homogeneous sums: universality of
Gaussian Wiener chaos},
Ann. Probab. 38 (2010), 1947--1985.

\bibitem[NP05]{NP05}
D.~Nualart, G.~Peccati,
\emph{Central limit theorems for sequences of multiple stochastic
integrals},
Ann. Probab. 33 (2005), 177--193.

\bibitem[PSTU10]{PSTU10}
G.~Peccati, J.L.~Sol\'e, M.S.~Taqqu, F.~Utzet,
\emph{Stein's method and normal approximation of Poisson functionals},
Ann. Probab. 38 (2010), 443--478.

\bibitem[PT11]{PT11}
G.~Peccati, M.S.~Taqqu,
\emph{Wiener Chaos: Moments, Cumulants and Diagrams},
Bocconi \& Springer Series, vol.~1,
Springer, Milan, 2011.

\bibitem[PT13]{PT13}
G.~Peccati, C.~Th\"ale,
\emph{Gamma limits and $U$-statistics on the Poisson space},
ALEA Lat. Am. J. Probab. Math. Stat. 10 (2013), 525--560.

\bibitem[PT05]{PT05}
G.~Peccati, C.A.~Tudor,
\emph{Gaussian limits for vector-valued multiple stochastic
integrals},
in: \emph{S\'eminaire de Probabilit\'es XXXVIII},
Lecture Notes in Mathematics, vol.~1857,
Springer, Berlin, 2005, pp.~247--262.

\bibitem[PZ10]{PZ10}
G.~Peccati, C.~Zheng,
\emph{Multi-dimensional Gaussian fluctuations on the Poisson space},
Electron. J. Probab. 15 (2010), 1487--1527.

\bibitem[PZ14]{PZ14}
G.~Peccati, C.~Zheng,
\emph{Universal Gaussian fluctuations on the discrete Poisson chaos},
Bernoulli 20 (2014), 697--715.

\bibitem[QSS07]{QSS07}
A.~Quarteroni, R.~Sacco, F.~Saleri,
\emph{Numerical Mathematics},
2nd ed., Texts in Applied Mathematics, vol.~37,
Springer, Berlin, Heidelberg, 2007.


\bibitem[Zhe18]{Zhe18}
G.~Zheng,
\emph{Recent developments around the Malliavin--Stein approach:
fourth-moment phenomena via exchangeable pairs},
Ph.D. thesis, Universit\'e du Luxembourg, 2018.


\bibitem[Zhe26a]{Zhe26a}
G.~Zheng,
\emph{A Kolmogorov fourth-moment bound on Poisson chaos via a
martingale core},
Preprint, arXiv:2607.28742, 2026.


\bibitem[Zhe26b]{Zhe26b}
G.~Zheng,
\emph{Four-moment criteria for Poisson convergence
on Poisson and Rademacher chaoses},
Preprint, arXiv:2608.12451, 2026.

\end{thebibliography}
\end{document}